\documentclass[11pt]{amsart}

\usepackage{amssymb,amsmath,enumerate,epsfig}
\usepackage{amsfonts,color}
\usepackage{a4,latexsym}
\usepackage{parskip}
\usepackage[all]{xy}

\usepackage{hyperref}
\hypersetup{pdftitle=Enriques}

\newcommand{\C} {\mathbb{C}}
\newcommand{\Q} {\mathbb{Q}}
\newcommand{\N}  {\mathbb{N}}

\newcommand{\Z}{\mathbb{Z}}

\newcommand{\OO}{\mathcal{O}}

\newcommand{\mX}{\mathcal{X}}
\newcommand{\mY}{\mathcal{Y}}

\newcommand{\PP}{\mathbb{P}}
\newcommand{\NS}{\mathop{\rm NS}}
\newcommand{\Num}{\mathop{\rm Num}}
\newcommand{\MW}{\mathop{\rm MW}}
\newcommand{\MWL}{\mathop{\rm MWL}}
\newcommand{\Km}{\mathop{\rm Km}}

\newcommand{\rank}{\mathop{\rm rank}}

\newcommand{\Br}{\mathop{\rm Br}\nolimits}

\newcommand{\te}{\tilde E}
\newcommand{\Aut}{\mathop{\rm Aut}}
\newcommand{\Hom}{\mathop{\rm Hom}}
\newcommand{\Triv}{\mathop{\rm Triv}}
\newcommand{\Gal}{\mathop{\rm Gal}}

\newcommand{\Fix}{\mathop{\rm Fix}}

\newcommand{\jj}{\jmath}

\newcommand{\mat}[3]{\begin{pmatrix} #1 & #2\\#2 & #3\end{pmatrix}}

\newtheorem{Theorem}{Theorem}[section]
\newtheorem{Proposition}[Theorem]{Proposition}
\newtheorem{Lemma}[Theorem]{Lemma}

\newtheorem{Corollary}[Theorem]{Corollary}

\theoremstyle{remark}
\newtheorem{Remark}[Theorem]{Remark}

\newtheorem{Example}[Theorem]{Example}

\theoremstyle{definition}
\newtheorem{Conjecture}[Theorem]{Conjecture}

\newtheorem{Exception}[Theorem]{Exception}
\newtheorem{Problem}[Theorem]{Problem}

\begin{document}

\title{Arithmetic of singular Enriques Surfaces}

\dedicatory{Dedicated to the memory of Eckart Viehweg}

\author{Klaus Hulek}
\address{Institut f\"ur Algebraische Geometrie, Leibniz Universit\"at
  Hannover, Welfengarten 1, 30167 Hannover, Germany} 
\email{hulek@math.uni-hannover.de}

\author{Matthias Sch\"utt}
\email{schuett@math.uni-hannover.de}



\thanks{Partial support from DFG under grant Hu 337/6-1 is gratefully acknowledged.}

\date{November 28, 2010}

\maketitle

 \begin{abstract}
 We study the arithmetic of Enriques surfaces whose universal covers are singular K3 surfaces.
 If a singular K3 surface $X$ has discriminant $d$,
 then it has a model over the ring class field $H(d)$.
 Our main theorem is that the same holds true for any Enriques quotient of $X$.
 It is based on a study of N\'eron-Severi groups of singular K3 surfaces.
 We also comment on Galois actions on divisors of Enriques surfaces.
 \end{abstract}

\section{Introduction}

Enriques surfaces have formed a vibrant research area over the last 30 years.
In many respects, they 
share the properties  of K3 surfaces, yet in other aspects they behave differently.
This twofold picture is illustrated in this paper which investigates arithmetic aspects of Enriques surfaces.

The arithmetic of Enriques surfaces is only partially well-understood.
For instance, Bogomolov and Tschinkel proved that
potential density of rational points holds on Enriques surfaces \cite{BT}.
The cited work predates all substantial progress on K3 surfaces in the same direction.
In fact, until now the corresponding statement for K3 surfaces has not been proved in full generality.

In this paper, we investigate the arithmetic of those Enriques surfaces whose universal covers are singular K3 surfaces, i.e.~K3 surfaces with Picard number $\rho=20$.
We will refer to them as \emph{singular Enriques surfaces}.
Singular K3 surfaces are closely related to elliptic curves with complex multiplication (CM).
These structures will be crucial to our investigations;
often they explain arithmetic properties of singular K3 surfaces (see Sections \ref{s:K3} and \ref{s:class}).

We point out one particular property that illustrates these relations:
the field of definition.
A singular K3 surface of discriminant $d$ has a model over the ring class field $H(d)$
just like elliptic curves with CM in an order of discriminant $d$ by \cite[Prop.~4.1]{S-fields}.
Our main theorem states how this property carries over to Enriques surfaces:

\begin{Theorem}
\label{thm}
Let $Y$ be an Enriques surface whose universal cover $X$ is a singular K3 surface.
Let $d<0$ denote the discriminant of $X$.
Then $Y$ admits a model over the ring class field $H(d)$.
\end{Theorem}

The proof of Theorem \ref{thm} consists in two steps:
first we establish a general result for automorphisms of K3 surfaces over number fields (Proposition \ref{Lem});
then we extend the afore-mentioned results for fields of definition of singular K3 surfaces 
to include their N\'eron-Severi groups (Theorem \ref{Thm:NS-H(d)}).
Here we combine two approaches that both rely on elliptic fibrations.
In Section \ref{s:K3} we review the theory of singular K3 surfaces
and use Inose's pencil and the theory of Mordell-Weil lattices 
to deduce Theorem \ref{Thm:NS-H(d)} for most singular K3 surfaces (see Remark \ref{Rem}).
On the other hand, Section \ref{s:Kummer} provides a direct approach for 
those singular K3 surfaces which are Kummer (Corollary \ref{Cor:Km}).
Through Shioda--Inose structures,
we then connect the two partial results 
and are thus able to give a full proof of Theorem \ref{Thm:NS-H(d)} (see \ref{ss:sum}).


In Section \ref{s:Enr} we  address explicit questions.
Lattice theoretically one can determine
all singular K3 surfaces that admit an Enriques involution.
With 61 or 62 exceptions,
we give an explicit geometric construction of an Enriques involution on these singular K3 surfaces.
This construction combines Shioda-Inose structures (\ref{ss:SI})
and the base change approach from \cite[\S 3]{HS}.

In Section \ref{s:class}
we discuss the problem of Galois action on N\'eron-Severi groups.
In this context, a different picture arises for Enriques surfaces than for K3 surfaces. 
The paper concludes with a formulation of several interesting classification problems for Enriques surfaces and K3 surfaces.

\smallskip

{\fontsize{9pt}{11pt}\selectfont

\noindent {\bfseries Acknowledgements}: 
We would like to thank Bas Edixhoven and Jaap Top for useful comments.
We are grateful to the referees for many helpful suggestions and remarks.
This project was started when the second author held a position at University of Copenhagen.
\par}

\section{Automorphisms of K3 surfaces}

\subsection{Basics about K3 surfaces and Enriques surfaces}

This paper is concerned with complex algebraic K3 surfaces and Enriques surfaces.
Here we briefly review their basic properties.
For details the reader is referred to \cite[Chapter VIII]{BHPV};
 information and examples relevant for this paper can also be found in \cite{HS}.

A \emph{K3 surface} $X$ is a smooth projective surface 
with trivial canonical bundle $\omega_X\cong\OO_X$ that is simply connected.
The classical example consists in a smooth quartic in $\PP^3$;
here we will mostly work with elliptic K3 surfaces and Kummer surfaces.

In terms of the Enriques--Kodaira classification of algebraic surfaces,
a complex \emph{Enriques surface} $Y$
is a smooth projective surface with vanishing irregularity $q(Y)=h^1(Y,\OO_Y)=0$ 
and $\omega_Y^{\otimes 2} =\OO_Y$, but $\omega_Y\neq\OO_Y$.
Equivalently $Y$is the quotient  of a K3 surface $X$ by a fixed point free involution $\tau$.
Conversely the K3 surface $X$ can be recovered as the universal covering of $Y$.


The \emph{N\'eron-Severi group} $\NS(S)$ of an algebraic surface $S$
is the group of divisors up to algebraic equivalence.
Here we identify divisors moving in families such as fibers of a fibration.
The N\'eron-Severi group is finitely generated abelian;
its rank is called the \emph{Picard number} and denoted by $\rho(S)$.
In essence, $\NS(S)$ encodes the discrete structure of the Picard group of $S$.
The intersection pairing endows $\NS(S)$ with a quadratic form
that also induces the notion of numerical equivalence.

On a K3 surface algebraic and numerical equivalence coincide,
and $\NS(S)$ is torsion-free.
Equipped with the intersection form, it becomes an even lattice of signature $(1,\rho(S)-1)$,
the \emph{N\'eron-Severi lattice}.
On an Enriques surface, however,
algebraic and numerical equivalence do not coincide,
as in $\NS(Y)$ there is two-torsion represented by the canonical divisor $K_Y$.
The quotient gives the torsion-free group of divisors up to numerical equivalence:
\[
\Num(Y) = \NS(Y)/\{0, K_Y\}.
\]
The intersection pairing endows $\Num(Y)$ with a lattice structure.
Contrary to the K3 case, this lattice has always the same rank and abstract shape:
\[
\Num(Y) = U + E_8(-1),\;\;\; \mbox{rank}(\Num(Y))=10
\]
where $U$ denotes the hyperbolic plane $\Z^2$ with intersection pairing $\begin{pmatrix}0&1\\1&0\end{pmatrix}$
and $E_8$ is the unique even unimodular positive-definite lattice of rank 8.
The $-1$ indicates that the sign of the intersection form is reversed
so that $\Num(Y)$ has signature $(1,9)$ as predicted by the Hodge index theorem.

The Torelli theorem \cite{PSS} reduces many investigations of complex K3 surfaces $X$ to a study of $H^2(X)$
with its different structures as lattice or Hodge structure.
By the cycle class map, $H^2(X)$ contains an algebraic part
coming from 
$\NS(X)$. 
The orthogonal complement of $\NS(X)$ in $H^2(X,\Z)$ is called the
\emph{transcendental lattice}:
\[
T(X) = \NS(X)^\bot\subset H^2(X,\Z).
\]
As another characterisation, $T(X)$ is the smallest primitive sublattice of $H^2(X,\Z)$ 
that contains the (up to scalar unique) 2-form $\eta_X$ after complexifying.

\subsection{Surfaces over number fields}

We will consider complex surfaces $S$ 
that admit a model over some number field.
This arithmetic setting brings up the natural question
whether geometric objects such as $\NS(S)$ or the automorphism group $\Aut(S)$
are defined over the same field. The problem is as follows:

Let $X$ be a complex K3 surface defined over a number field $L$.
The action of its absolute Galois group  $G_L=\text{Gal}(\bar L/L)$ on $\NS(X)$ factors through a finite extension $M/L$.
We say that $\NS(X)$ is defined over $L$ if $M=L$, i.e.~if $G_L$ acts trivially on $\NS(X)$.
Throughout this paper, we will verify this property by exhibiting a set of generators of $\NS(X)$ each of which is defined over $L$.
In fact, for elliptic surfaces with section (which we will mostly be concerned with), both conditions are equivalent.
%

The same terminology is employed for an Enriques surface $Y$ 
by saying that $\NS(Y)$ or $\Num(Y)$ is defined over a number field $L$
if $G_L$ acts trivially.


Let $\psi$ be an automorphism of a complex K3 surface $X$.
Since we assumed $X$ to be algebraic,
the induced automorphism $\psi^*$ acts as multiplication by a root of unity $\zeta$ on the holomorphic 2-form $\eta_X$.
We assume that $X$ is defined over some number field.
The next proposition gives a criterion for the field of definition of $\psi$.
This criterion will be crucial for the proof of Theorem \ref{thm}.

\begin{Proposition}
\label{Lem}
Let $X$ be a K3 surface over some number field $L$.
Let $\psi\in\Aut(X)$ and $\zeta\in\bar\Q$ such that $\psi^*\eta_X=\zeta\eta_X$.
Assume that $\NS(X)$ is defined over $L$ and $\zeta\in L$.
Then $\psi$ is defined over $L$.
\end{Proposition}

\begin{proof}
We first need to show that $\psi$ is defined over some number field.
Essentially this holds true 
because the automorphism group of any algebraic K3 surface is discrete by \cite[Thm.~0.1]{Sterk}.
The general idea is well-known:
if the field of definition of $\psi$ were to require a transcendental extension of $L$,
then the transcendental generators of this extension could be turned into parameters,
so that $\psi$ would come in a non-discrete family of automorphisms.

Now suppose that $\psi$ is defined over some finite extension $M/L$.
We want to apply the Torelli theorem \cite{PSS} to $\psi$ and its conjugates to deduce that $M=L$. 
For this purpose, we assume without loss of generality that $M/L$ is Galois.
Let $\sigma\in\mbox{Gal}(M/L)$.
Then $\psi^\sigma\in\Aut(X)$, and we claim that $\psi=\psi^\sigma$.
Explicitly we can write 
\[
\psi^\sigma=\sigma\circ\psi\circ\sigma^{-1}.
\]
By the Torelli theorem, it suffices to verify the claim for the induced action on $\NS(X)$ and $T(X)$.
For $\NS(X)$ this follows directly from the fact that 
$\sigma$ and $\sigma^{-1}$ act trivially by assumption.
For $T(X)$, it suffices to check the action on the holomorphic $2$-form.
One has
\[
(\psi^\sigma)^*(\eta_X) = 
(\sigma^{-1})^*\circ\psi^*(\eta_X) =
(\sigma^{-1})^*(\zeta\eta_X)
 = \zeta^\sigma \eta_X = \psi^*(\eta_X)
\]
since $\zeta\in L$.
Hence $\psi^* = (\psi^\sigma)^*$ on $H^2(X,\Z)$, and the claim $\psi=\psi^\sigma$ follows from the Torelli theorem \cite{PSS}.
In consequence, $\psi$ is defined over $L$.
\end{proof}

\begin{Remark}
The conditions of Proposition \ref{Lem} are sufficient, but not necessary.
For instance, we exhibited a K3 surface with an Enriques involution over $\Q$,
but with $\NS(X)$ defined over $\Q(\sqrt{-3})$ in \cite[\S 5.3]{HS} (see also \ref{ss:Gal}).
\end{Remark}

\subsection{Enriques involutions}

Proposition \ref{Lem} has an immediate impact on involutions, and in particular on Enriques involutions.
Namely for an involution $\psi$, the eigenvalue of $\eta_X$ can only be $\zeta=\pm 1$, 
so Proposition \ref{Lem} only requires the N\'eron-Severi group of the covering K3 surface to be defined over $L$:

\begin{Corollary}
\label{Cor}
Let $X$ be a K3 surface over some number field $L$.
If $\NS(X)$ is defined over $L$,
then so is every involution on $X$.
In particular, this holds for Enriques involutions.
\end{Corollary}

Theorem \ref{thm} requires some concepts that we will discuss in detail in the next section.
It concerns K3 surfaces with Picard number 20, the so-called \emph{singular K3 surfaces} (see \ref{ss:sing}).
By definition, the discriminant of a singular K3 surface $X$ 
is the determinant of the intersection form on $\NS(X)$.
For a singular K3 surface, the discriminant $d$ gives rise to a very particular number field,
the ring class field $H(d)$ as we discuss in \ref{ss:K3-fields}.
In order to deduce Theorem \ref{thm},
it suffices to combine Corollary \ref{Cor} 
with the following result for any singular K3 surface (admitting an Enriques involution):

\begin{Theorem}
\label{Thm:NS-H(d)}
Let $X$ be a singular K3 surface of discriminant $d$.
Consider the ring class field $H(d)$.
Then $X$ admits a model over $H(d)$ with $\NS(X)$ defined over $H(d)$.
\end{Theorem}

The statement about a model over the ring class field $H(d)$ has been known before 
(cf.~\cite[Prop.~4.1]{S-fields}),
but the extension for the N\'eron-Severi group seems to have gone unnoted until now.
A proof will be given in the next two sections after reviewing the previous relevant results on singular K3 surfaces.
We conclude this section with a direct corollary:


\begin{Corollary}
\label{Cor:Num}
Let $Y$ be an Enriques surface whose universal cover $X$ is a singular K3 surface.
Let $d<0$ denote the discriminant of $X$.
Then $Y$ admits a model over the ring class field $H(d)$ with $\Num(Y)$ defined over $H(d)$.
\end{Corollary}

The corresponding statement for $\NS(Y)$ does not hold true in general,
as we will discuss within the framework of Galois actions on divisors in \ref{ss:NS}
(see Example \ref{Ex:Gal}, Corollary \ref{Lem:B}).

\section{Arithmetic of singular K3 surfaces}
\label{s:K3}

This section will review those parts of the theory of singular K3 surface that are relevant to our issues.
The section culminates in Lemma \ref{Lem:MW1},
the main step towards the proof of Theorem \ref{Thm:NS-H(d)}.
It is based on Shioda-Inose structures and Inose's fibration.
All the required techniques will be explained along the way.

\subsection{Singular K3 surfaces}
\label{ss:sing}

A complex K3 surface $X$ is called \emph{singular} if its Picard number $\rho(X)=\rank \NS(X)$ equals the maximum number allowed by Lefschetz' theorem:
\[
\rho(X) = h^{1,1}(X) = 20.
\]
Singular K3 surfaces involve no moduli,
so the terminology "singular" should be understood in the sense of exceptional
(just like for singular j-invariants of elliptic curves with complex multiplications,
a similarity that will become clear very soon).
We will discuss fields of definition of singular K3 surfaces in \ref{ss:K3-fields}.
Recently singular K3 surfaces over $\Q$
have gained some prominence due to modularity;
namely, in analogy with the Eichler-Shimura correspondence 
between modular forms of weight 2 and elliptic curves over $\Q$,
for any suitable modular form of weight 3 there is a singular K3 surface over $\Q$ associated
 (cf.~\cite{ES}).


By the Torelli theorem \cite{PSS}, \cite{SI},
singular K3 surfaces are classified up to isomorphism by their 
transcendental lattices.
For a singular K3 surface, the transcendental lattice is even and positive definite of rank two
and endowed with an orientation.
Up to conjugation in $\mbox{SL}_2(\Z)$, we identify it with the quadratic intersection form
\begin{eqnarray}
\label{eq:Q}
Q(X) = \begin{pmatrix}
2a & b\\
b & 2c
\end{pmatrix}
\end{eqnarray}
with integer entries $a,c\in\N, b\in\Z$ and discriminant $d=b^2-4ac<0$.
This number equals the determinant of the intersection form on $\NS(X)$;
we refer to it as the \emph{discriminant} of $X$.
By the Torelli theorem \cite{PSS}, \cite{SI} two singular K3 surfaces are isomorphic 
if and only if the transcendental lattices admit an isometry preserving the orientation 
(or equivalently the quadratic forms are conjugate in $\mbox{SL}_2(\Z)$).

The classical example for a singular K3 surface is the Fermat quartic in $\PP^3$.
Here we give an alternative example in terms of an elliptic  fibration that will reappear 
later in this paper in another context (\ref{ss:exc}).
Our treatment draws on the theory of elliptic surfaces;
all relevant concepts can be found in  \cite{SSh} for instance.

\begin{Example}
\label{Ex1}
Consider the universal elliptic curve for $\Gamma_1(6)$:
\[
\mathcal E: \;\; y^2 + (t-2)xy-t(t-1)y = x^3 - tx^2.
\]
Here a point of order six is given by $(0,0)$.
$\mathcal E$ gives rise to a rational elliptic surface $S$ over $\PP^1$.
By Tate's algorithm \cite{Tate}, $S$ has the following singular fibres in Kodaira's notation:
$$
\begin{array}{cccccc}
\hline
\text{fibre} && I_6 & I_3 & I_2 & I_1\\
\hline
t && \infty & 0 & 1 & -8\\
\hline
\end{array}
$$
Any quadratic base change $f$ of $\PP^1$ gives rise to a K3 surface $X$.
Generally $\rho(X)\geq 18$ by the Shioda--Tate formula \cite[Cor.~5.3]{ShMW}, 
but one can increase the Picard number conveniently 
by infering ramification points at singular fibres.
For instance, setting $t=-8s^2/(s^2-1)$ yields an elliptic K3 surface $X$ with three singular fibres of type $I_2$ and $I_6$ each, and thus $\rho(X)=20$ over $\C$ again by the Shioda-Tate formula
and the Lefschetz inequality $\rho(X)\leq h^{1,1}(X)$.
On $X$, there are two additional two-torsion sections with $x$-coordinate $-4s^2(3s\pm 1)(s\mp 1)/(s^2-1)^2$.
General theory shows that 
the singular fibres do not allow any further torsion in the Mordell-Weil group.
Over $\C$ one obtains $\MW(X)=\Z/2\Z\times \Z/6\Z$..
It follows that $X$ is the universal elliptic curve for the group $\Gamma_1(6)\cap\Gamma(2)$.
By  \cite[11.10 (22)]{SSh},
$\NS(X)$  has discriminant $-12$.
With  the discriminant form \`a la Nikulin \cite[Prop.~1.6.1 \& Cor.~1.9.4]{N}, one can then compute the transcendental lattice 
with intersection form $Q(X)=\mbox{diag}(2,6)$ (in agreement with the tables in \cite{SZ}).
\end{Example}


\subsection{Shioda-Inose structure}
\label{ss:SI}

In order to prove the surjectivity of the period map,
mathematicians first considered Kummer surfaces.
However, singular abelian surfaces (with $\rho(A)=4$) cannot possibly yield all singular K3 surfaces
as Kummer surfaces
because the transcendental lattice of a  Kummer surface is always two-divisible as an even lattice.
In detail, the intersection form is obtained from $T(A)$ by multiplication by $2$:
\[
T(\Km(A)) = T(A)(2).
\]
This problem of non-primitivity was overcome
by Shioda and Inose in \cite{SI}.
Generally they considered two elliptic curves $E, E'$.
Their product is an abelian surface $A=E\times E'$ and yields the Kummer surface $X'=\Km(E\times E')$.
Over $\C$, the Picard numbers depend on whether $E$ and $E'$ are isogenous ($E\sim E'$) or have complex multiplication (CM):
\begin{eqnarray}\label{eq:rho}
\rho(A)=\begin{cases}
2, & \text{if $E\not\sim E'$;}\\
3, & \text{if $E\sim E'$ without CM};\\
4, & \text{if $E\sim E'$ with CM.}
             \end{cases}
             \;\;\;\;\;
             \rho(X') = \rho(A) + 16.
\end{eqnarray}

The Kummer surface $X'$ admits several jacobian elliptic fibrations.
For instance, the projections onto the factors $E$ and $E'$ induce two isotrivial elliptic fibrations on the Kummer surface $X'$ that we will analyse in Section \ref{s:Kummer}.
In \cite[\S2]{SI}, a jacobian elliptic fibration with a fibre of type $II^*$ was found on $X'$.
It has exactly two further reducible fibres of the following types:
$$
\begin{array}{lcl}
2 I_0^* && E\not\cong E',\\
I_0^*, I_1^* && E\cong E', j(E)\neq 0, 12^3,\\
2 I_1^* && j(E)=j(E')=12^3,\\
I_0^*, IV^* && j(E)=j(E')=0.
\end{array}
$$
Starting from this elliptic fibration, we proceed with the quadratic base change 
\[
f: \PP^1\to\PP^1
\]
that ramifies exactly at the above two reducible singular fibres.
Since both ramified fibres are non-reduced, 
the base change applied to $X'$ results in another elliptic K3 surface $X$.
By construction, the elliptic K3 surface $X$ has two fibres of type $II^*$ and possibly some reducible fibres of type $I_2$ or $IV$ depending on the above cases.
The Kummer surface $X'$ can be recovered from $X$ as (the desingularisation of) the quotient 
by the involution of the double cover $X\dasharrow X'$.
(In \cite{HS} we abused terminology by referring to this involution as deck transformation,
but here we will call it base change involution.)
The base change involution is a Nikulin involution that composes the involution on the base curve $\PP^1$
 with the hyperelliptic involution on the fibres:
 \[
  \xymatrix{A \ar@{-->}[dr] && X\ar@{-->}[dl]\\
 & \Km(A)=X'&}
 \]
%
%
The gist of this construction is that the K3 surface $X$ recovers the transcendental lattice of the abelian surface $A$:
\begin{eqnarray}
\label{eq:TT}
T(X) = T(X')(1/2) = T(A).
\end{eqnarray}
Morrison coined the terminology \textbf{Shioda-Inose structure} for such a setting: 
abelian surface and K3 surface with the same transcendental lattice
such that Kummer quotient and Nikulin involution yield the same Kummer surface.
He developed lattice theoretic criteria to decide which K3 surfaces of Picard number $\rho\geq 17$ admit a Shioda-Inose structure \cite[\S6]{Mo}.

\subsection{Surjectivity of the period map}

The surjectivity of the period map requires to exhibit singular K3 surfaces for any quadratic form $Q$ as in \eqref{eq:Q}.
By the above considerations,
this can be achieved by exhibiting a singular abelian surface $A$ with $Q(A)=Q$
because then the Shioda-Inose structure provides a suitable singular K3 surface $X$ with $Q(X)=Q$.

Chronologically, the corresponding surjectivity statement for singular abelian surfaces was already
established  before Shioda--Inose's work by Shioda and Mitani in \cite{SM}.
Namely, it was shown that any singular abelian surface has product type.
Given the quadratic form $Q(A)$ with coefficients as in \eqref{eq:Q}, the abelian surface $A$ admits the representation $A=E\times E'$
with the following elliptic curves given as complex tori $E_\tau = \C/(\Z+\tau\Z)$:
\begin{eqnarray}\label{eq:E}
E=E_\tau,\;\;\tau = \dfrac{-b+\sqrt{d}}{2a},\;\;\;\;\; E'=E_{\tau'},\;\;\tau' = \dfrac{b+\sqrt{d}}2.
\end{eqnarray}
Note that this representation need not be unique, and in fact there can be arbitrarily many distinct representations for the same singular abelian surface (and thus also for singular K3 surfaces). 

\begin{Example}
\label{Ex2}
The K3 surface $X$ from \ref{Ex1} is not a Kummer surface,
since $T(X)$ is not two-divisible as an even lattice.
Through the Shioda-Inose structure, $X$ arises from the self-product of the elliptic curve $E_{\sqrt{-3}}$ with j-invariant $2^43^35^3$.
\end{Example}

\subsection{Fields of definition}
\label{ss:K3-fields}

We have seen that 
every singular abelian surface $A$ is the product of two elliptic curves with CM in the same field.
CM elliptic curves are 
well-understood thanks to the connection to class field theory (cf.~\cite[\S 5]{Shimura}).
Indeed both curves in \eqref{eq:E} are defined over the ring class field $H(d)$.
This field is an abelian Galois extension of the imaginary quadratic field $K=\Q(\sqrt{d})$
with prescribed ramification and  Galois group isomorphic to the class group $Cl(d)$
(see \cite[\S 9]{C}).
We recall one way to describe $Cl(d)$: 
it consists of $\mbox{SL}_2(\Z)$-conjugacy classes of primitive $2\times 2$ matrices $Q$ as in \eqref{eq:Q} of discriminant $d<0$ together with Gauss composition (cf.~\cite[\S3]{C} for instance).
By \cite[Thm.~5.7]{Shimura},
$H(d)$ is generated over $K$ by adjoining the j-invariant of $E'$, 
or in fact of any elliptic curve with CM by the given  order in $K$ of discriminant $d$.
Here $Cl(d)$ acts naturally as a permutation on all these CM elliptic curves
-- abstractly on the complex tori, 
but also in a compatible way through the Galois action on $H(d)$ permuting j-invariants.

Shioda--Inose used these CM properties to deduce that any singular K3 surface is defined over some number field.
Namely, the Kummer quotient $X'$ respects the base field (a property that we will exploit in Section \ref{s:Kummer}).
Hence the only step in the Shioda-Inose structure that may require increasing the base field concerns 
the elliptic fibration with a fibre of type $II^*$.

Subsequently Inose exhibited an explicit model for $X$ over a specific extension of $H(d)$ in \cite{Inose}.
This model is expressed purely in terms of the j-invariants $j, j'$ of the elliptic curves $E, E'$ from \eqref{eq:E}:
\begin{eqnarray}\label{eq:Inose}
X:\;\;\; y^2 = x^3 - 3At^4x + t^5(t^2 - 2 B t+1),
\end{eqnarray}
where $A^3 = jj'/12^6,\; B^2 = (1-j/12^3)(1-j'/12^3)$.
Thus we know that any singular K3 surface $X$ of discriminant $d$
admits a model over a degree six extension of $H(d)$.
In \cite[Prop.~4.1]{S-fields} it was then noted that the above fibration can be twisted in such a  way that it is defined over $H(d)$ (cf.~\eqref{eq:WF-X2} in case $AB\neq 0$):

\begin{Theorem}
\label{Thm:H(d)}
Let $X$ be a singular K3 surface  of discriminant $d$.
Then $X$ has a model over the ring class field $H(d)$.
\end{Theorem}

In practice, the given field of definition can be far from optimal,
that is, $X$ may admit a model over a much smaller number field.
In fact, the modularity converse in \cite{ES} required to exhibit models of singular K3 surfaces over $\Q$
where the ring class field had degree as large as $32$ over $\Q$.
We can already detect a similar behaviour on the level of the elliptic curves $E, E'$ in \eqref{eq:E}:
because of the Galois action of the class group $Cl(d)$, the elliptic curve $E'$ can at best be defined over a quadratic subfield of $H(d)$.
The factor $E$, however, may be defined over $\Q$ even for large $d$ by inspection of the denominators in \eqref{eq:E}.

\subsection{N\'eron-Severi group}

In the remainder of this section, we derive an important intermediate result for the
proof of Theorem \ref{Thm:NS-H(d)}.
The remaining steps will be done in Section \ref{s:Kummer} (cf.~\ref{ss:sum}).
We have recalled in Theorem \ref{Thm:H(d)} that any singular K3 surface $X$ admits a model 
over the ring class field $H(d)$.
Here $d$ denotes the discriminant of $T(X)$ as usual.
It remains to show that there always is a model of $X$ with $\NS(X)$ defined over $H(d)$ as well.


The basic idea for the proof is to work with a model of Inose's pencil \eqref{eq:Inose} over $H(d)$
as in the proof of \cite[Prop.~4.1]{S-fields}:
\begin{eqnarray}
\label{eq:model}
X: \;\;\; y^2 = x^3 + at^4x + t^5(b_2t^2+b_1t+b_0),\;\;\; a, b_i\in H(d).
\end{eqnarray}
Note that fibres of type $II^*$ do not admit any inner Galois action (i.e.~on fibre components).
Hence these two singular fibres of $X$ together with the zero section generate a sublattice $U+2E_8(-1)\subset\NS(X)$ that is fully defined over the base field $H(d)$.
It remains to study the Galois action on the remaining generators of $\NS(X)$ (there are two generators remaining, since $\rho(X)=20$).
Looking at the other reducible singular fibres, we distinguish four cases as in \ref{ss:SI}:

\begin{table}[ht!]
$$
\begin{array}{lcl}
\hline
\text{Reducible fibres other than $II^*$} & \mbox{rank}(\MW) & \text{case}\\
\hline
- & 2 & E\not\cong E',\\
I_2 & 1 & E\cong E', j(E)\neq 0, 12^3,\\
2I_2 & 0 & E\cong E', j(E)=12^3,\\
IV & 0 &E\cong E', j(E)=0.\\
\hline
\end{array}
$$
\caption{Singular fibres and $\MW$-rank of Inose's pencil}
\label{tab:MW}
\end{table}

\begin{Lemma}
\label{Lem:cases}
If the singular K3 surface $X$ admits an Inose pencil \eqref{eq:Inose}
of $\MW$-rank at most one,
then $X$ has a model with $\NS(X)$ defined over $H(d)$.
\end{Lemma}

\begin{proof}
For the last two surfaces in Table \ref{tab:MW} ($\MW$-rank zero), there are explicit models with $\NS(X)$ defined over $\Q$ (cf.~\cite[\S10]{S-NS}).
For the case of $\MW$-rank one with an $I_2$ fibre, it is also easy to see that $\NS(X)$ can be defined over $L=H(d)$.
The fibre does not admit any Galois action, since the identity component is fixed by Galois.
By the formula of Shioda-Tate, the Mordell-Weil group has rank one.
The Mordell-Weil generator $P$ can only be either fixed or  mapped to its inverse by Galois.
But if the latter is the case, then the section $P$ is defined 
over some quadratic extension of $L$.
More precisely, it is given in $x,y$-coordinates
as $P=(U,\sqrt\gamma V)$  for some $\gamma\in L, U,V\in L(t)$.
Consider the quadratic twist of $X$ with respect to this quadratic extension of $L$:
\[
\gamma y^2 = x^3 + at^4x + t^5(b_2t^2+b_1t+b_0).
\]
This is an alternative model of the fixed elliptic fibration \eqref{eq:model} on $X$ over $L$
such that both models become isomorphic over $L(\sqrt\gamma)$. 
This quadratic twist transforms the section to $(U,V)$ (defined over $L$) without introducing any Galois action on the singular fibres (since they only have types $I_1, I_2, II, II^*$).
Thus the N\'eron-Severi group of the new model of $X$ is defined over $L=H(d)$.
\end{proof}

\begin{Remark}
If $T(X)$ is primitive and lies in the principal genus,
then it is possible to replace the CM-curves $E, E'$ by opposite Galois conjugates that are isomorphic:
$E^\sigma\cong (E')^{\sigma^{-1}}$.
By \cite[\S6]{S-fields} (which combines \cite{Shimura} and  \cite{SM}), 
one has $T(E^\sigma\times (E')^{\sigma^{-1}}) = T(E\times E')$.
According to Table \ref{tab:MW}, the induced Inose pencil on $X$ has $\MW$-rank one.
By Lemma \ref{Lem:cases} this  produces a model of $X$ with $\NS(X)$ defined over $H(d)$.
\end{Remark}

\subsection{Mordell-Weil lattices}
\label{ss:MWL}

A similar argument goes through for almost all instances of the case where $E\not\cong E'$.
Here we can argue with the Mordell-Weil lattice $\MWL(X)$ of the fibration.
In general, the Mordell-Weil lattice of an elliptic surface $S\to C$ with section 
was defined by Shioda in \cite{ShMW} as follows.
In $\NS(S)$ consider the trivial lattice $\Triv(S)$ generated by the zero section and fibre components.
By \cite[Thm.~1.3]{ShMW} there is an isomorphism 
\[
\MW(S)\cong \NS(S)/\Triv(S).
\]
The torsion in $\MW(S)$ is contained in (and determined by) the primitive closure $\Triv(S)'$ of $\Triv(S)$ inside $\NS(S)$.
The quotient $\MW(S)/\MW(S)_\text{tor}$ is endowed with a lattice structure 
by means of the orthogonal projection $\varphi$ in $\NS(S)_\Q$ with respect to $\Triv(S)$.
Here tensoring with $\Q$ is required unless $\Triv(S)'$ is unimodular.
By construction $\varphi(\MW(S))(-1)$ is a positive definite, though not necessarily integral lattice 
that one refers to as \emph{Mordell-Weil lattice} $\MWL(S)$.
The Mordell-Weil lattice satisfies functorial properties for base change and Galois actions.
For details the reader is referred to \cite{ShMW} or the survey paper \cite{SSh}.

In the present situation the only reducible fibres have type $II^*$.
The non-identity fibre components generate the root lattice $E_8(-1)$,
so $\Triv(X) = U+2E_8(-1)$.
Hence
$\MWL(X)$ is a positive definite even integral lattice of rank two that fits into the decomposition
\[
\NS(X)=U+2E_8(-1)+\MWL(X)(-1).
\]
Since $\Triv(X)$ is unimodular,
the discriminant forms of $\NS(X)$ and $\MWL(X)$ agree up to sign.
By \cite[Cor.~1.9.4]{N}, this implies that $T(X)$ and $\MWL(X)$ lie in the same genus (or in the same isogeny class).

\subsection{Binary even quadratic forms}
\label{ss:bin}

To understand the possible Galois actions on $\MWL(X)$, we shall need a simple observation about the automorphisms of such lattices.
It will be phrased in terms of the corresponding quadratic form $Q$ as in \eqref{eq:Q}. 
Multiplication by $\pm 1$ gives the trivial automorphisms of $Q$;
any other automorphism will be called non-trivial.
The problem whether $Q$ admits non-trivial automorphisms depends on its order in the class group
of even positive definite binary quadratic forms with given discriminant and degree of primitivity:

\begin{Lemma}
The positive-definite quadratic form $Q$ admits a non-trivial automorphism
if and only if it is two-torsion in its class group.
\end{Lemma}


The proof is elementary, so we will omit it here although we did not find a concise reference.
For later use, we shall give the possible automorphism groups.
Recall that any quadratic form $Q$ as in \eqref{eq:Q} can be transformed by conjugation in $SL_2(\Z)$ to a reduced form where the coefficients satisfy
$-a<b\leq a\leq c$ (and $b\geq 0$ if $a=c$).
The inverse of a quadratic form is obtained by replacing $b$ by $-b$.
A reduced quadratic form is two-torsion if and only if
\[
b=0 \;\;\; \text{ or } \;\;\; a=b \;\;\; \text{ or } \;\;\; a=c.
\]
We obtain the following non-trivial automorphism groups where $D_{2n}$ denotes the dihedral group of order $2n$:
\begin{table}[ht!]
$$
\begin{array}{c|ccccc}
\hline
Q & \mat{2a}{0}{2c} & \mat{2a}{a}{2c} & \mat{2a}{b}{2a} & \mat{2a}{0}{2a} & \mat{2a}{a}{2a}\\
 & a<c & a<c & 0<b<a &&\\
\hline
\Aut(Q) & (\Z/2\Z)^2 & (\Z/2\Z)^2 & (\Z/2\Z)^2 & D_8 & D_{12}\\
\hline
\end{array}
$$
\caption{Quadratic forms with non-trivial automorphisms groups}
\label{Tab:aut}
\end{table}

\subsection{Intermediate step}
\label{ss:EE'}

We conclude this section with an intermediate result towards the proof of Theorem \ref{Thm:NS-H(d)}.
In the next section, we will use the Shioda-Inose structure to complete the proof.

\begin{Lemma}
\label{Lem:MW1}
 In all cases of $\MW$-rank two in Table \ref{tab:MW}, the model \eqref{eq:Inose}
admits a twist  such that there is an $H(d)$-rational section.
\end{Lemma}

\begin{proof}
If the automorphism group of $\MWL$ is only two-torsion, 
then the lemma follows after a quadratic twist for one of the $\MW$ generators.
This leaves the cases of the last two quadratic forms in Table \ref{Tab:aut}.
Here the class number of $Q$ is one.
Hence $T(X)$ has exactly the intersection form $Q$.
In the Shioda-Inose structure, we can choose $E$ by \eqref{eq:E} with j-invariant $j=12^3$ resp.~$j=0$.
The extra automorphism of $E$ induces an extra automorphism on $X$ that respects the elliptic fibration \eqref{eq:Inose}:
\[
 (x,y,t) \mapsto (-x, iy, -t) \;\;\; \text{resp.~}\;\; (x,y,t) \mapsto (\varrho x,y,t)
\]
where $\varrho, i$ denote primitive third resp.~fourth roots of unity.
The respective automorphism makes $\MWL(X)$ into a module of rank one over $\Z[i]$ resp.~$\Z[\varrho]$.
This identification is compatible with the Galois action over $H(d)$, since the automorphisms are defined over $H(d)$.
Hence it suffices to study the Galois action on the given modules of rank one.
Their only automorphisms are the units in $\Z[i]$ resp.~$\Z[\varrho]$,
i.e.~the group of fourth resp.~sixth roots of unity.
On the elliptic curves with CM by these rings,
it is well-known that such a Galois action can be accounted for by biquadratic or sextic twisting (see \cite[\S II, Example 10.6 \& Exercises 2.33, 2.34]{Si} or \cite[\S8]{S-MMJ}).
Thanks to the special shape of the present Weierstrass form \eqref{eq:Inose} with $A=0$ or $B=0$,
this translates directly into twists of $X$.
Thus there is a twist with $\MWL(X)$ defined over $H(d)$.
\end{proof}

\begin{Remark}
\label{Rem}
If $\MWL$ admits no non-trivial automorphisms,
then Lemma \ref{Lem:MW1} already settles Theorem \ref{Thm:NS-H(d)} completely.
By the proof of Lemma \ref{Lem:MW1},
this also holds for $\MWL$ with non-abelian automorphism group (the last two entries in Table \ref{Tab:aut}).
It is the two-torsion cases of Table \ref{Tab:aut} that require an extra argument.
\end{Remark}

In the next section, we will use the Shioda-Inose structures and study Kummer surfaces of product type in detail.
In this case, although
we may not have any automorphisms on the Kummer surface to relate the $\MW$-generators,
we can 
use the endomorphisms of the abelian surface instead.
This approach will enable us to complete the proof of Theorem \ref{Thm:NS-H(d)} in \ref{ss:sum}.

\section{Singular Kummer surfaces of product type}
\label{s:Kummer}

Let $E, E'$ be isogenous complex elliptic curves with CM.
Then the abelian surface $A=E\times E'$ is singular ($\rho(A)=4)$).
Let $d$ denote its discriminant (that is the discriminant of $T(A)$).
Then $E, E'$ have models over the ring class field $H(d)$ (obtained from the CM-field by adjoining the j-invariants).

Throughout this section, we only consider the case where $E\not\cong E'$ ($\MW$-rank two)
and no j-invariant equals $0$ or $12^3$ (no extra automorphisms).
The same results hold in the other cases, but we would have to distinguish more subcases
and also consider biquadratic/sextic twisting etc.
Note that for the excluded cases we have already given a full proof of Theorem \ref{Thm:NS-H(d)}
in Lemma \ref{Lem:cases} (for $E\cong E'$) and in the proof of Lemma \ref{Lem:MW1} (for $j$ or $j'\in\{0,12^3\}$; cf.~Remark \ref{Rem}).
Thus the cases considered explicitly in this section will suffice to complete the proof of Theorem \ref{Thm:NS-H(d)}.

\subsection{}

Consider the Kummer surface $X'=\Km(A)$.
Recall the isotrivial elliptic fibrations on $X'$  
that are induced by the projections onto $E$ and $E'$ from \ref{ss:SI}.
These are naturally defined over $H(d)$ as follows.
Fix Weierstrass models 
\begin{eqnarray}
\label{eq:WF-E}
 E:\;\; y^2 = f(x),\;\;\;\; E': \;\; y^2 = g(x)
\end{eqnarray}
with cubic polynomials $f, g\in H(d)[x]$.
Then $X'$ admits a birational model
\begin{eqnarray}
\label{eq:Km}
 X':\;\; f(t) y^2 = g(x)
\end {eqnarray}
with the structure of an elliptic curve over the function field $H(d)(t)$.
We denote the corresponding elliptic fibration by the pair $(X', \pi)$.
This fibration has singular fibres of type $I_0^*$ at $\infty$ and at the zeroes of $f(t)$.
Over $\bar\Q$ we have $\MW(X',\pi)=\Z^2\times (\Z/2\Z)^2$ with torsion sections given by the roots of $g(x)$.

\begin{Proposition}
\label{Prop:Km-MW}
The elliptic fibration $(X',\pi)$ admits a model over $H(d)$ such that $\MW$ is generated by two-torsion and sections defined over $H(d)$.
In particular, $\MWL$ is generated by sections defined over $H(d)$.
\end{Proposition}

\begin{proof}
By the Shioda-Tate formula, the Mordell-Weil lattice has rank two since $\rho(X')=20$.
Due to the singular fibre types $\MWL(X',\pi)$ will not be integral, but it is positive-definite.
Hence the results from \ref{ss:bin}, \ref{ss:EE'} apply directly to prove the claim with the exception of the first three special cases from Table \ref{Tab:aut}.
Here we pursue an alternative uniform approach based on the fact that as in Lemma \ref{Lem:MW1} we can find a quadratic twist with at least one $\MW$-generator $P$ over $H(d)$.

The crucial ingredient is the following lattice isomorphism which Shioda established in \cite[Prop.~3.1]{Sh-Murre}:
\begin{eqnarray}
\label{eq:Hom}
 \Hom(E, E') \cong \MWL(X',\pi).
\end{eqnarray}
Here $\Hom(E, E')$ is endowed with a norm given by the degree.
The isomorphism takes a homomorphism $\phi: E \to E'$ as input.
Via its graph $\Gamma_\phi$ in $A$ and the image $\bar\Gamma_\phi$ in $X'$,
one associates to $\phi$ the element $\bar R_\phi$ in $\MWL(X',\pi)$ corresponding to $\bar\Gamma_\varphi$ under the orthogonal projection $\NS(X')\to\MWL(X',\pi)$ (see \ref{ss:MWL}).

In \cite{Sh-Murre} Shioda worked over an algebraically closed field,
so that the isomorphism \eqref{eq:Hom} is independent of the chosen model.
However, for the specified models in \eqref{eq:WF-E}, \eqref{eq:Km}
 the isomorphism \eqref{eq:Hom} is clearly Galois-equivariant.

Following Lemma \ref{Lem:MW1}, 
we apply a quadratic twist on $X'$ such that
there is an $H(d)$-rational section $P$ (non-torsion).
That is, for some $c\in H(d)$ we consider the $H(d)(\sqrt{c})$-isomorphic model
\[
X':\;\; cf(t)y^2=g(x).
\]
In terms of the elliptic curves $E, E'$, this is accounted for by twisting \emph{one} elliptic curve by $\sqrt{c}$, say:
\begin{eqnarray}
\label{eq:WF-E'}
 E:\;\; y^2 = f(x),\;\;\;\; E': \;\; cy^2 = g(x).
\end{eqnarray}
For these models, the isomorphism \eqref{eq:Hom} is by construction again Galois-equivariant.
Hence the section $P$ corresponds to a homomorphism $\phi: E\to E'$ over $H(d)$.
Now pick any endomorphism $\epsilon$ of $E'$ that is not multiplication by an integer.
By CM-theory, $\epsilon$ is defined over $H(d)$, and together $\phi, \epsilon\circ\phi$ generate the lattice $\Hom(E,E')$ up to finite index.
In conclusion, \eqref{eq:Hom} gives a section $R_{\epsilon\circ\phi}$ over $H(d)$ that is independent of $P$.
By construction, these sections generate $\MWL(X',\pi)$ up to finite index.
Proposition \ref{Prop:Km-MW} thus follows.
\end{proof}

\subsection{N\'eron-Severi group of Kummer surfaces}

We collect a few consequences of Proposition \ref{Prop:Km-MW}.
We start with  a version of Theorem \ref{Thm:NS} for singular Kummer surfaces.
Note that since $T(X')=T(A)(2)$, the Kummer surface $X'$ has discriminant $4d$.

\begin{Corollary}
\label{Cor:Km}
The singular Kummer surface $X'$ has a model over $H(d)$ with $\NS(X')$ defined over $H(4d)$.
\end{Corollary}

\begin{proof}
Fix the model of the elliptic fibration $(X',\pi)$ from Proposition \ref{Prop:Km-MW} with $\MW$-rank two over $H(d)$.
In order to generate $\NS(X')$, we have to add to these $H(d)$-rational sections the two-torsion sections and the components of the $I_0^*$ fibres.
These rational curves are defined over the splitting field of the polynomials $f(t), g(x)$ over $H(d)$.
That is, we adjoin to $H(d)$ the $x$-coordinates of the two-torsion points of $E$ and $E'$.
By the analogue of the Kronecker-Weber theorem for imaginary quadratic number fields \cite[\S II Thm.~5.6]{Si}, these algebraic numbers generate exactly $H(4d)$ over $H(d)$.
\end{proof}

\subsection{Isogenous CM-elliptic curves}

Before continuing with the proof of Theorem \ref{Thm:NS-H(d)},
we note another implication of Proposition \ref{Prop:Km-MW}. 
Here we are concerned with the field of definition of the isogeny between $E$ and $E'$.
By the classical theory, any two elliptic curves with CM in the same field $K$
have models over some minimal ring class field $H$;
moreover they are isogenous over $\bar\Q$.
Here we ask whether they admit $H$-isogenous models, i.e.~models over $H$ 
with isogeny defined over $H$ as well.
When the CM-curves are $\Q$-curves, this property comes for free, but this situation does not always persist (cf.~Remark \ref{Rem:QQ}).
The following result might be well-known to the experts, but we could not find a reference.

\begin{Corollary}
\label{Cor:iso-E}
Let $E, E'$ be elliptic curves with CM by orders in the same imaginary quadratic field $K$.
Let $H=K(j(E), j(E'))$.
Then $E, E'$ have $H$-isogenous models.
\end{Corollary}

\begin{proof}
We can start with any two Weierstrass forms over $H$ as in \eqref{eq:WF-E}.
The proof of Proposition \ref{Prop:Km-MW} exhibits a quadratic twist of $E'$ with a non-trivial homomorphism $\phi: E\to E'$.
\end{proof}

\begin{Remark}
\label{Rem:QQ}
Corollary \ref{Cor:iso-E} only seemingly conflicts with a result of Gross \cite[\S11]{G}.
Namely, Gross found that there are CM-elliptic curves which are not $\Q$-curves,
i.e.~$E$ is not $H$-isogenous to all its conjugates.
Here we let $E'=E^\sigma$ be a conjugate of $E$.
If $E, E^\sigma$ are not $H$-isogenous (so that~$E$ is not a $\Q$-curve),
then Corollary \ref{Cor:iso-E} provides us with a quadratic twist of $E^\sigma$ which is $H$-isogenous to $E$.
But then the quadratic twist of $E^\sigma$ and $E$ are not conjugate anymore,
so there is no contradiction to $E$'s failure of being a $\Q$-curve.
\end{Remark}

\subsection{Auxiliary elliptic fibration}
\label{ss:aux}

Recall the singular K3 surface $X$ with Inose's elliptic fibration \eqref{eq:Inose}.
By \cite{Sandwich} the quadratic base change $t=u^2$ recovers the Kummer surface $X'$.
Since $X$ also dominates $X'$ by the Shioda-Inose structure, Shioda alluded to this picture 
as $X$ being sandwiched by the Kummer surface $X'$.
In the base change, the two fibres of type $II^*$ are replaced by fibres of type $IV^*$.
Let us explain how to find this base changed fibration on the previous model of $X'$:
\[
X':\;\;\; cf(t)y^2 = g(x).
\]
Projection onto the affine coordinate $u=y$ endows $X'$ with the structure of an elliptic fibration $\pi'$
since the fibres are plane cubics in $x,t$.
Write $(X',\pi')$ for $X'$ with this fixed elliptic fibration.
Visibly $(X',\pi')$ is the quadratic base change of the rational elliptic surface $S'$ obtained by setting $u^2=v$.
$S'$ has singular fibres of type $IV$ at $v=0,\infty$;
in $X'$ they are replaced by fibres of type $IV^*$ as alluded to before.
Here $S'$ is given as a cubic pencil whose base points form sections.
Recall that these sections are all defined over $H(4d)$.

By base change $\MWL(S)(2)$ embeds into $\MWL(X',\pi')$.
Consider the orthogonal complement
\[
L=[\MWL(S')(2)]^\bot\subset\MWL(X',\pi').
\]
By construction,
$L$ is exactly the invariant sublattice of $\MWL(X',\pi')$ for the involution corresponding to the base change $X'\to X$, i.e.~$L=\MWL(X)(2)$.

Over $\bar\Q$ (or in fact algebraically closed fields of characteristic $\neq 2,3$),
Shioda used a similar argument as for the isomorphism \eqref{eq:Hom} to derive an isomorphism
\begin{eqnarray}
\label{eq:isos}
L\cong \Hom(E,E')(4), \;\;\; \text{so that}\;\;\; \MWL(X)\cong\Hom(E,E')(2).
\end{eqnarray}
Compared to the previous argument that gave \eqref{eq:Hom}, there is one subtlety here:
For $\phi\in\Hom(E,E')$, the orthogonal projection onto $L_\Q$ maps the divisor $\bar\Gamma_\phi$
to $\frac 12 L$.
This holds true since the quotient $\MWL(X',\pi')/(L+L^\bot)$ need not be trivial (hence we tensor $L$ with $\Q$ a priori),
but due to the quadratic base change the quotient is always isomorphic to a finite number of copies of $\Z/2\Z$.
Now instead of $\bar\Gamma_\phi$, one takes the image of the divisor $2\bar\Gamma_\phi$ in $L$.
Computing intersection numbers using the theory of Mordell-Weil lattices, Shioda verifies the isomorphism \eqref{eq:isos}.
In our setting, the main problem is to find models which make the isomorphisms \eqref{eq:isos} Galois-equivariant over a suitable field.

\subsection{Galois-equivariance}
\label{ss:Gal-eq}

We know that $E, E'$ admit $H(d)$-isogenous models, so that $\Hom(E,E')$ is generated by isogenies over $H(d)$.
The elliptic fibration $\pi'$ on $X'$ is defined over $H(d)$ as well,
but in order to endow it with a section (a base point of the cubic pencil), we may have to increase the base field to $H(4d)$.
This makes the isomorphisms in \eqref{eq:isos} for the specified models Galois-equivariant over $H(4d)$.
For $X$, however, we need a model with $\MWL$ over $H(d)$, so we have to throw in some more information.
We distinguish two cases according to the degree $h$ of the Galois extension $H(4d)/H(d)$.
Note that with the Legendre symbol $(\cdot/2)$ at $2$, 
one obtains from the class number formula
\[
h=\deg(H(4d)/H(d))=
\begin{cases}
1, & (d/2)=1 \text{ or } d=-3, -4;\\
2, & 2\mid d, \; d\neq -4;\\
3, & (d/2)=-1, \; d\neq -3.
\end{cases}
\]

\subsubsection{First case: $h=1,2$}

This case is very simple. 
By assumption, both polynomials $f, g$ have a root over $H(d)$.
A base point of the cubic pencil gives  an $H(d)$-rational section of the elliptic fibration $(X',\pi')$.
Due to the singular fibre types and the involution $u\mapsto -u$, we obtain a Weierstrass form
\begin{eqnarray}
\label{eq:WF-X'}
X':\;\; y'^2 = x'^3 - 3au^4 x' + u^4 (b_2u^4-2b_1u^2+b_0).
\end{eqnarray}
As quotient by the base change involution $u\mapsto -u$ of $X'\to S'$ composed with the hyperelliptic involution $y'\mapsto -y'$, we obtain a model of $X$ over $H(d)$.
Compared to \eqref{eq:Inose}, this Weierstrass form is not yet normalised with respect to $b_0, b_2$.

By construction, the isomorphisms \eqref{eq:isos} are $H(d)$-Galois equivariant for these specific models of $E, E', X', X$.
That is, we have exhibited a model of $X$ over $H(d)$ with fibration of type \eqref{eq:Inose} and $\MW$-rank two over $H(d)$.
It follows that this model has $\NS(X)$ defined over $H(d)$.

\subsubsection{Second case: $h=3$}

In this case, we compare two $\bar\Q$-isomorphic models that we denote by $X_1, X_2$.
From \eqref{eq:WF-X'},
we obtain a model over $H(4d)$
as quotient by the Nikulin involution $(x',y',u)\mapsto (x',-y',-u)$:
\begin{eqnarray}
\label{eq:WF-X1}
X_1:\;\; y'^2 = x'^3 - 3au^4 x' + u^5 (b_2u^2-2b_1u+b_0)
\end{eqnarray}
with $\MWL(X_1)$ defined over $H(4d)$ by the Galois-equivariant isomorphism \eqref{eq:isos}.
From \eqref{eq:Inose}, we derive a model over $H(d)$
\begin{eqnarray}
\label{eq:WF-X2}
X_2:\;\; y^2 = x^3 - 3c^2B^2A^3t^4 x + c^3B^2A^3t^5 (B^2t^2-2B^2t+1).
\end{eqnarray}
Here $B^2,A^3\in H(d)$ as given in \ref{ss:SI}.
By Lemma \ref{Lem:MW1}, we can choose $c\in H(d)$ in such a way that $X_2$ has an $H(d)$-rational section $P$ and an orthogonal section $Q$ defined over some quadratic extension $M$ of $H(d)$.
We assume that $M\neq H(d)$ and derive a contradiction from the above two models. 
Essentially, this works because we compare a quadratic and a cubic extension of $H(d)$.

By assumption, we can choose $Q$ anti-invariant under conjugation in $M/H(d)$
(so that $P,Q$ generate $\MW(X_2)$ up to finite index).
Hence there are rational functions $x_Q, y_Q\in H(d)(t)$ and some constant $c_Q\in H(d)$ such that 
\[
Q=(x_Q, \sqrt{c_Q} y_Q) \;\;\; \text{ and } \;\;\; M=H(d)(\sqrt{c_Q}).
\]
We work out an isomorphism of the two elliptic fibrations $X_1, X_2$.
This can only take the shape
\begin{eqnarray}
\label{eq:X12}
(x,y,t) \mapsto (x',y',u) = (\gamma\alpha^2 x, \alpha^3\gamma^{3/2} y, \alpha t).
\end{eqnarray}
Thus we require
\[
a=\gamma^2 (c^2B^2A^3), \; b_1=\gamma^3(c^3B^4A^3), \; \alpha b_2=\gamma^3(c^3B^4A^3), \; b_0=\alpha\gamma^3 (c^3B^2A^3).
\]
The first two relations give
$\gamma=b_1/(acB^2)\in H(4d)$, so that also $\alpha\in H(4d)$.
The section $P$ on $X_2$ with $H(d)$-rational $y$-coordinate $y_P(t)$ pulls back to a section $P_1$ with $y'$-coordinate $\gamma^{3/2}\alpha^3y_P(\alpha t)$. 
By construction, $P_1$ is $H(4d)$-rational, so $\gamma^{3/2}\in H(4d)$.
But here $H(4d)$ has degree three over $H(d)$, so $\gamma^{3/2}\in H(d)$.
In other words, the isomorphism \eqref{eq:X12} is defined over $H(4d)$.

In consequence, $Q$ pulls-back to a section on $X_1$ with $y'$-coordinate $\sqrt{c_Q}$ times an $H(4d)$-rational function.
The same argument as for $\gamma^{3/2}$ then shows that $\sqrt{c_Q}\in H(d)$. 
This gives the required contradiction.



\subsection{Proof of Theorem \ref{Thm:NS-H(d)}}
\label{ss:sum}

We collect all results necessary to prove Theorem \ref{Thm:NS-H(d)}.
Let $X$ be a singular K3 surface of discriminant $d$.
We decided to work with Inose's pencil over $H(d)$ as in \eqref{eq:WF-X2}.
Thus it suffices to check the field of definition of $\MW(X)$ to verify Theorem \ref{Thm:NS-H(d)}.
In many cases, this was achieved in Lemma \ref{Lem:cases} or in the intermediate Lemma \ref{Lem:MW1} (as explained in Remark \ref{Rem}).
For the remaining K3 surfaces, we considered the Kummer surface $X'$ from the Shioda-Inose structure
that actually sandwiches $X$ (\ref{ss:aux}).
Note that for Kummer surfaces we exhibited a proof of Theorem \ref{Thm:NS-H(d)} that only uses the techniques from Lemma \ref{Lem:MW1} (Proposition \ref{Prop:Km-MW}, Corollary \ref{Cor:Km}).
Thanks to the interplay between $H(d)$ and $H(4d)$, this suffices to deduce that $\MW(X)$ is defined over $H(d)$ by \ref{ss:Gal-eq}.
This completes the proof of Theorem \ref{Thm:NS-H(d)}.
\qed


\section{Enriques surfaces of base change type}
\label{s:Enr}

This section provides a technique to construct explicit examples of Enriques surfaces whose covers are singular K3 surfaces.
In the sequel, we refer to them as \emph{singular Enriques surfaces}.
The main idea is to invoke the base change construction from \cite[\S3]{HS} for singular K3 surfaces.
We will review the concept in \ref{ss:HS} and then relate it to the Shioda-Inose structures from \ref{ss:SI}.


\subsection{Singular K3 surfaces with Enriques involution}

Our first problem concerns K3 surfaces:
Which singular K3 surfaces admit an Enriques involution?
Keum's result \cite{Keum} gives a partial answer for all singular K3 surfaces
that are Kummer surfaces
(i.e.~with transcendental lattice two-divisible).
The full problem can also be solved by purely lattice-theoretic means in terms of the transcendental lattice.
In fact, one finds that the discriminant almost suffices to reach a decision:
it suffices for non-Kummer surfaces
while for Kummer surfaces we know the answer anyway from \cite{Keum}.
Sert\"oz gave the solution in \cite{Sert},
based on the techniques developed by Keum \cite{Keum}:

\begin{Theorem}
\label{Thm:cover}
Let $X$ be a singular K3 surface of discriminant $d$.
Then $X$ does not admit an Enriques involution exactly in the following cases:
\begin{enumerate}[(i)]
\item
$d\equiv -3\mod 8$,
\item
$d= -4,-8$,
\item
$d=-16$ and $X$ is not Kummer, i.e.~$Q(X)=\mbox{diag}(2,8)$.
\end{enumerate}
\end{Theorem}
Note that the discriminants in case $(ii)$ determine unique singular K3 surfaces up to isomorphism.
In case $(iii)$, we have to exempt the Kummer surface $\Km(E_i\times E_i)$ 
with transcendental lattice of intersection form $Q=\mbox{diag}(4,4)$
which admits an Enriques involution by \cite{Keum}.

Sert\"oz' proof is purely lattice theoretic and based on machine computations.
In particular,
for those singular K3 surfaces admitting some Enriques involution,
it does not give any explicit geometric description of any such involution.
Here we shall combine the ideas from \cite[\S3]{HS} and Section \ref{s:K3}
to derive explicit Enriques involutions on almost all singular K3 surfaces possible according to Theorem \ref{Thm:cover}.

\subsection{Enriques involutions of base change type}
\label{ss:HS}

We start by reviewing the set-up from \cite[\S3]{HS}:

\begin{table}[ht!]
\begin{tabular}{lcl}
$S$ && Rational elliptic surface\\
$f$ && quadratic base change of $\PP^1$ (not ramified at non-reduced fibres of $S$)\\
$X$ && base change of $S$ by $f$: K3 surface\\
$\imath$ && base change involution\\
$(-1)$ && hyperelliptic involution\\
 $\boxplus P$ && translation by a section $P\in\MW(X)$\\
\end{tabular}
\end{table}

In this situation, the composition $\jj=\imath\circ(-1)$ defines a Nikulin involution on $X$, 
i.e.~$\jj$ has eight isolated fixed points and leaves the holomorphic two-form invariant.
The quotient $X/\jj$ has a resolution $X'$ that is again K3.
$X'$ is the quadratic twist of $S$ at the ramification points of the base change $f$:
The induced action of $\imath$ and $\jj$ gives a decomposition of $\MW(X)$ up to some $2$-power index:
\begin{eqnarray}
\label{eq:MWQ}
\MW(X)_\Q \cong \MW(S)_\Q + \MW(X')_\Q.
\end{eqnarray}
Let $P'\in\MW(X')$ and $P$ denote the induced section on $X$.
By construction, $P$ is anti-invariant for $\imath^*$.
In consequence, $\boxplus P\circ\imath$ defines an involution $\tau$ on $X$.
By definition, this involution can only have fixed points on the fixed fibres of $\imath$.
If these fibres are smooth, one has
\[
 \Fix(\tau) = \emptyset \; \Longleftrightarrow \; P\cap O\cap\Fix(\imath)=\emptyset.
\]
The latter condition can be checked with $P'$ on the ramified fibres of $X'$ (generally of type $I_0^*$).
Here $P'$ has to meet non-identity components.

\begin{Example}
The prototype example for this construction is a two-torsion section $P$ induced from $X'$ 
(or equivalently from $S$ since two-torsion is not affected by quadratic twisting).
Outside characteristic two, such a section is always disjoint from $O$.
For $\tau$ to have fixed points,
one of the ramified fibres has to be singular such that it is additive or $P$ meets the identity component.

The latter occurs for Example \ref{Ex1}:
There is exactly one two-torsion section induced from $S$.
This section $(t-1,t-1)$ meets both ramified fibres (at $0$ and $\infty$) at their identity components.
The other two-torsion sections are interchanged by $\imath$ 
(which is why \eqref{eq:MWQ} only holds after tensoring with $\Q$).
\end{Example}

\subsection{}
We ask which singular K3 surfaces admit an Enriques involution of base change type.
For now we only exclude 62 or 63 singular K3 surfaces as specified in Exception \ref{exc}
(62 assuming some special cases of ERH, see \ref{ss:1-cong}).

\begin{Proposition}
\label{Prop:sing-Enr}
Let $X$ be a singular K3 surface admitting an Enriques involution.
Assume that $X$ is not among the 62 or 63 K3 surfaces from Exception \ref{exc}.
Then $X$ has an Enriques involution $\tau$ of base change type
where the Nikulin quotient $X'$ is a Kummer surface.
\end{Proposition}

The proof of the proposition will be given in sections \ref{ss:1-cong} and \ref{ss:2}.
It is based on the Shioda-Inose structure of singular K3 surfaces
to that we will return next.

One word about Exception \ref{exc}: we do not believe this exception to be necessary,
but we have not found a general argument to overcome it (cf.~Remark \ref{Rem:exc}).
To illustrate this, we will show in \ref{ss:exc}
that Example \ref{Ex1} which falls under Exception \ref{exc} 
does indeed admit an Enriques involution of base change type
(but we did not check whether the quotient $X'$ is a Kummer surface).

\subsection{Enriques involutions and Shioda-Inose structures}
\label{ss:SI2}

Let $E, E'$ denote elliptic curves 
and consider the corresponding Shioda-Inose structure as in \ref{ss:SI}.
Then $X'=\Km(E\times E')$ admits an Enriques involution by \cite{Keum},
but how about the K3 surface $X$ from \ref{ss:SI}
that recovers the transcendental lattice of the abelian surface $E\times E'$?

If $E$ and $E'$ are not isogenous, then $X$ has Picard number $\rho(X)=18$ and the fibration \eqref{eq:Inose} of Mordell-Weil rank zero yields
\[
\NS(X)= U + 2E_8(-1).
\]
This lattice does not admit any primitive embedding of the Enriques lattice $U(2)+E_8(-2)$ because of the $2$-length.
Hence the K3 surface $X$ cannot have an Enriques involution.
We now consider the case where $E$ and $E'$ are isogenous, possibly with CM.

Here is our main tool to construct explicit Enriques involutions: 
the Shioda-Inose structure falls under the settings studied in \ref{ss:HS}.
We already chose the notation to indicate this:
there is a K3 surface $X$ with a Nikulin involution yielding the Kummer surface $X'$.
Conversely, $X$ is obtained from $X'$ by a quadratic base change.
In terms of the elliptic fibration \eqref{eq:Inose} on $X$,
the Nikulin involution is given as
\[
\jj: \; (x,y,t) \mapsto (x/t^4, -y/t^6, 1/t).
\]
Thus the quotient $X/\jj$  attains singularities in the fibres at $t=\pm 1$
whose minimal resolution is $X'$.
In general, the quotient results in fibres of type $I_0^*$,
but there are other possibilities as sketched in \ref{ss:SI}.
Concretely, there is another involution corresponding to the base change $\PP^1\to\PP^1$ induced by $X\to X'$:
\[
\imath=\jj\circ(-1): (x,y,t) \mapsto (x/t^4, y/t^6, 1/t).
\]
The quotient $X/\imath$ gives a rational elliptic surface $S$.
It extends the Shioda-Inose structure to the following diagram
(where we could also add the induced elliptic fibrations):
 \[
  \xymatrix{E\times E' \ar@{-->}[dr] && X\ar@{-->}[dl] \ar@{-->}[dr]&\\
 & \Km(E\times E')=X'&&S}
 \]
%
By construction, $S$ has a singular fibre of type $II^*$.
From the Shioda-Tate formula \cite[Cor.~5.3]{ShMW}, it follows that $S$ is extremal, i.e.~it has finite Mordell-Weil group.
Since a singular fibre of type $II^*$ does not admit any torsion sections (of order relatively prime to the characteristic),
we infer that $\MW(S)=\{O\}$.
By \cite[Prop.~8.12]{ShMW} (cf.~\eqref{eq:MWQ}), this implies that 
\[
\MWL(X)=\MWL(X')(2).
\]
Hence as soon as the Mordell-Weil rank of $X$ is positive,
there is a section $P$ (induced from $X'$) and an involution $\tau$ as in \ref{ss:HS}.
In order to exhibit an Enriques involution on $X$, 
it remains to determine whether $\tau$ is fixed point free.
In general there are three cases of positive Mordell-Weil rank to 
be distinguished according to the types of singular fibres.
For non-singular K3 surfaces, i.e.~Mordell-Weil rank one with $E\not\cong E'$ and $\rho=19$,
this has been done in \cite[\S 4.2]{HS} (without referring to Shioda-Inose structures).
The property whether $\tau$ is fixed point free or not depends on the parity of the height of the Mordell-Weil generator modulo $4$.
In the next sections, we will treat the singular cases and thus prove Proposition \ref{Prop:sing-Enr}.


\begin{Remark}
There is a natural continuation of this connection between Enriques involutions of base change type and  Shioda-Inose structures.
Recall from Section \ref{s:Kummer}
that the K3 surface $X$ is sandwiched by the Kummer surface $X'$
in the following sense:
$X'$ can also be recovered from $X$ by the quadratic base change $u\mapsto t=u^2$ applied to \eqref{eq:Inose}.
As in \ref{ss:HS}, each section of $X$ induces an involution $\tau$ of base change type
on the Kummer surface $X'$.
Here we ask whether $\tau$ is an Enriques involution.
We have seen that the base change replaces the fibres of type $II^*$ by type $IV^*$ (so these are fixed by $\tau$).
However, none of these fibre types admits a free involution,
so there cannot be an Enriques involution on $X'$ as in \ref{ss:HS} for the specified base change.
\end{Remark}

\subsection{Mordell-Weil rank one and $E\cong E'$}
\label{ss:1-cong}

In this case, $E$ is a CM elliptic curve with $j(E)\neq 0, 12^3$.
The elliptic fibration \eqref{eq:Inose} on $X$ has exactly one reducible fibre of type $I_2$ at $t=1$ 
in addition to the two fibres of type $II^*$. 
Together with the Mordell-Weil generator $P$, we can write
\[
\NS(X) = U + 2E_8(-1) + \langle A_1(-1), P\rangle.
\]
We consider two cases according to the intersection behaviour of the section $P$ and the fibre of type $I_2$.

If $P$ meets the non-identity component of the $I_2$ fibre,
then $P$ has height
\[
h(P) = 4 + 2(P\cdot O) -1/2.
\]
Equivalently, the discriminant $d=-2h(P)$ of $X$ is odd.
Clearly, $P$ and $O$ do not intersect on the $I_2$ fibre which is one of the two fixed fibres of the base change involution $\imath$.
Here translation by $P$ exchanges the fibre components including  the nodes,
so it acts freely on the singular fibre.
It remains to check for the specialisation of $P$ on the other fixed fibre at $t=-1$.
Note that $P$ is induced from a section $P'$ on the Nikulin quotient $X'$, so
\[
P\cdot O = 2(P'\cdot O') + \# P\cap O \cap\mbox{Fix}(\imath).
\]
Since $P$ and $O$ can only possibly intersect on the irreducible fixed fibre of $\imath$ at $t=-1$,
the parity of the intersection number $P\cdot O$ depends only the intersection behaviour at that fibre.
In consequence,
the discriminant $d$ of $X$ satisfies the congruence
\[
d\equiv -7\mod 8 \Longleftrightarrow P\cap O \cap\mbox{Fix}(\imath)=\emptyset \Longleftrightarrow \mbox{Fix}(\tau)=\emptyset.
\]
In comparison, Theorem \ref{Thm:cover} states that a singular K3 surface of odd discriminant $d$
admits an Enriques involution if and only if $d\equiv -7\mod 8$.
This proves Proposition \ref{Prop:sing-Enr} for all odd discriminants and $\MW$ rank one cases.
(As explained  in \ref{ss:EE'} such fibrations exist on $X$~if and only if the transcendental lattice is primitive and lies in the principal genus.)

We now consider the case where $X$ has even discriminant, i.e.~
the section $P$ meets the identity component of the $I_2$ fibre.
Then $\tau$ fixes both fibre components.
As they are isomorphic to $\PP^1$, 
there are fixed points.
(In fact one can see that $\tau$ fixes one component pointwise.)
In conclusion, the given elliptic fibration \eqref{eq:Inose} on $X$ does not admit an Enriques involution
of base change type.

This failure to produce an Enriques involution poses the problem how 
it can be overcome for the singular K3 surfaces in consideration for Proposition \ref{Prop:sing-Enr}.
Recall that we are in the special case where the fibration \eqref{eq:Inose} corresponds to $E\cong E'$.
The principal idea now is to choose an alternative elliptic  fibration of the same kind on $X$, 
but for a pair $(E,E')$ such that  $E\not\cong E'$
(resembling our approach in \ref{ss:EE'}).
Whenever this is possible, the new fibration falls under the next case of Mordell-Weil rank two,
and Proposition \ref{Prop:sing-Enr} can be proved along those lines.
Here we can vary the pair $(E, E')$ by conjugates $(E^\sigma, (E')^{\sigma^-1})$.
This fails to return a fibration of $\MW$ rank two if and only $E^\sigma\cong E^{\sigma^{-1}}$ for all Galois elements $\sigma$.
Equivalently, the class group is only two-torsion.
Note that $E\cong E'$ implies that $T(E\times E')=T(X)$ is primitive and lies in the principal genus.
Since  the same applies to all conjugates,
we derive the following abstract characterisation of the singular K3 surfaces where the Shioda-Inose structure does not produce an Enriques involution of base change type:

\begin{Exception}
\label{exc}
A singular K3 surface $X$ of even discriminant $d$ does not admit an elliptic fibration \eqref{eq:Inose}
of Mordell-Weil rank two 
if and only  if $T(X)$ is primitive and gives  the full principal genus of its class group.
In other words~$Q(X) = \mbox{diag}(2,|d|/2)$ and the class group $Cl(d)$ is only two-torsion.
\end{Exception}

There are 101 known discriminants $d<0$ such that $Cl(d)$ is only two-torsion;
the discriminant of biggest absolute value is $d=-7392$.
By \cite{Wb}, there could be one more such discriminant of size $>10^{10}$,
but this is ruled out by the extended Riemann hypothesis for odd real Dirichlet
characters.
Out of the 101 known discriminants,
$65$ are even (they were already studied by Euler, cf.~\cite{C})
and $-4, -8, -16$ are ruled out by Theorem \ref{Thm:cover},
so the above exception concerns 62 or 63 singular K3 surfaces.
We consider one of them in detail in \ref{ss:exc} after completing the proof of Proposition \ref{Prop:sing-Enr}..

\begin{Remark}
\label{Rem:exc}
For each of the 62 known singular K3 surfaces from Exception \ref{exc}, 
one could try to exhibit an Enriques involution as in \ref{ss:HS} for a different base change 
than in the Shioda-Inose structure.
However, there does not seem to be a universal way to achieve this.
Notably, the general K3 surface $X$ arising from the Shioda-Inose structure for the present case $E\cong E'$ only admits four essentially different jacobian elliptic fibrations.
To see this, one can argue with a gluing technique of Kneser-Witt 
that has been successfully applied to K3 surfaces by Nishiyama in \cite{Nishi}.
For these four fibrations, the fibre types reveal that only \eqref{eq:Inose} and one other fibration can arise through a quadratic base change.
The latter pulls back from the unique rational elliptic surface 
with a singular fibre of type $I_9$ and $\MW=\Z/3\Z$
by the one-dimensional family of quadratic base changes that ramify at the reducible fibre.
A case-by-case analysis (exactly as above) shows that a singular elliptic K3 surface within this family 
can only have an Enriques involution of base change type if it does not fall under Exception \ref{exc}.
\end{Remark}

\subsection{Mordell-Weil rank two}
\label{ss:2}

In this case, $E$ and $E'$ are isogenous, but non-isomorphic elliptic curves with CM.
Both fixed fibres for the base change involution $\imath$ at $t=\pm 1$ are smooth.
On the Nikulin quotient $X'$, they correspond to fibres of type $I_0^*$.
As explained, the fibration \eqref{eq:Inose} on $X$ has integral even Mordell-Weil lattice $\MWL(X) = \MWL(X')(2) = \mbox{Hom}(E,E')(2)$,
and 
\[
\NS(X) = U + 2E_8(-1) + \MWL(X)(-1).
\]
For an Enriques involution $\tau$ on $X$, 
we ask that some section $P\in\MWL(X)$ meets both fixed fibres
at non-identity components.
Equivalently,
there is a section $P'\in\MWL(X')$ (inducing $P$) 
that meets both  ramified fibres (type $I_0^*$)
at non-identity components.

\textbf{Assumption:}
There is no such section $P'\in\MWL(X')$.
Equivalently, since the simple components of a fibre admit a group structure,
the non-identity components of one of the $I_0^*$ fibres are fully avoided by $\MW(X')$.
Correspondingly, $\NS(X')$ admits an orthogonal summand $D_4(-1)$ which we single out in the following decomposition:
\[
\NS(X') = U + E_8(-1) + D_4(-1) + \langle D_4(-1), \MWL(X')(-1)\rangle.
\]
Hence the discriminant group of $\NS(X')$ contains two copies of $\Z/2\Z$ (coming from $D_4^\vee/D_4$).
Indeed, since the length is bounded by the rank of the transcendental lattice, i.e.~by two,
this gives the full $2$-part of the discriminant group:
\begin{eqnarray}
\label{eq:5.6}
2\text{-part}(\NS(X')^\vee/\NS(X')) \cong D_4^\vee/D_4\cong (\Z/2\Z)^2.
\end{eqnarray}
Right away, we deduce that $\NS(X')$ has discriminant $d'$ equalling four times an odd integer.
By \eqref{eq:TT}, this odd integer is exactly~the discriminant  $d=d'/4$ of $X$.
In particular, if $d$ is even, 
$\MWL(X')$ cannot fully avoid the non-identity components of either of the $I_0^*$ fibres.
Thus there is a  section of the fibration \eqref{eq:Inose} inducing an Enriques involution $\tau$ on $X$.

To complete the proof of Proposition \ref{Prop:sing-Enr},
we return to the case of odd discriminant $d$.
The isomorphism \eqref{eq:5.6} gives an equality of discriminant forms
\[
-q_{D_4} = q_{D_4} = \left(q_{\NS(X')}\right)|_{2\text{-part}}.
\]
By \cite{N},
there is an equality $q_{\NS(X')} = -q_{T(X')}$.
Hence it suffices to compare the discriminant forms of $T(X')$ and $D_4$.
In the present situation, $T(X')$ has the quadratic form
\[
\begin{pmatrix}
4a & 2b\\
2b & 4c
\end{pmatrix}
\]
with odd $b$.
Hence its discriminant form takes the following values on a set of representatives of the $2$-part of $T(X')^\vee/T(X')$:
\[
0,a,c,a+b+c\mod 2\Z.
\]
In comparison, $q_{D_4}$ does exclusively attain the value $1$ mod $2\Z$ on the non-zero elements of $D_4^\vee/D_4$.
For $T(X')$, this can only happen if all $a,b,c$ are odd.
Equivalently, the discriminant satisfies $d\equiv-3\mod 8$.
This is exactly the main case excluded by Theorem \ref{Thm:cover}.

Conversely, we deduce that a singular K3 surface $X$ admits an Enriques involution 
if it has an elliptic fibration \eqref{eq:Inose} of Mordell-Weil rank two and 
if either $d$ is even or $d\equiv -7\mod 8$.
The latter can be achieved unless $T(X)$ is primitive and corresponds to the principal class in its class group which is only two-torsion (cf.~Exception \ref{exc}).
This completes the proof of Proposition \ref{Prop:sing-Enr}.

\subsection{Appendix: Example \ref{Ex1}}
\label{ss:exc}

In this paragraph, we will show that the singular K3 surface $X$ from Example \ref{Ex1}
(which falls under Exception \ref{exc})
does admit an alternative elliptic fibration
with an Enriques involution of base change type.
We will pursue an abstract approach following ideas of Kneser-Witt as worked out for elliptic K3 surfaces by Nishiyama \cite{Nishi}.

\begin{Lemma}
$X$ has an elliptic fibration with $\Z/3\Z\subset\MW$ and two fibres of type $I_9$.
\end{Lemma}

\begin{proof}
By \cite[\S6]{Nishi}
the elliptic fibrations on $X$ are classified by primitive embeddings of a certain partner lattice $M$ of $T(X)$ into Niemeier lattices.
Here we can take $M=A_1(-1) + A_5(-1)$ since $M$ and $T(X)$ have the same discriminant form.
Consider the Niemeier lattice $N$ with root lattice 
\[
N_\text{root}=A_8(-1)^3\;\;\; \text{ and quotient } \;\;\; N/N_\text{root} = (\Z/3\Z)^3.
\]
Embedding $M$ primitively into one summand $A_8(-1)$,
we obtain the essential lattice of an elliptic fibration of $X$ as orthogonal complement $M^\bot\subset N$.
The singular fibres of this fibration are encoded in the roots of $M^\bot$, i.e.~in $(M^\bot)_\text{root}=A_8(-1)^2$.
The torsion in $\MW$ for this fibration is isomorphic to the quotient of the primitive closure of $(M^\bot)_\text{root}$ in $N$ by $(M^\bot)_\text{root}$, i.e.~$\MW_\text{tor}\cong\Z/3\Z$.
\end{proof}

The given elliptic fibration is not isotrivial due to the singular fibres of type $I_9$.
The torsion in $\MW$ then implies that $X$ is a base change of the universal elliptic curve with $3$-torsion section and j-invariant not identical zero.
This elliptic surface has singular fibres $I_1, I_3, IV^*$,
so necessarily the base change factors through the intermediate rational elliptic surface $S'$ with configuration $I_1, I_1, I_1, I_9$ and $\MW(S')=\Z/3\Z$.
In particular, $X$ arises from $S'$ by a quadratic base change.
Hence we are in the set-up of \ref{ss:HS} with base change involution $\imath$ etc.

Now we consider the quadratic twist $X'$.
It is the desingularisation of the quotient of $X$ by the Nikulin involution $\jj=\imath\circ(-1)$.
We claim that this quotient exhibits another Shioda-Inose structure on $X$:

\begin{Lemma}
$X'$ is a Kummer surface with $T(X')=T(X)(2)$.
\end{Lemma}

\begin{proof}
It suffices to prove that $\jj$ is a Morrison-Nikulin involution, 
i.e.~$\jj^*$ exchanges two copies of $E_8(-1)$ in $\NS(X)$.
Here we argue with the above elliptic fibration:
$\jj$ exchanges the two reducible fibres of type $I_9$ and the three-torsion sections $Q, \boxminus Q$.
Consider these 20 rational curves on $X$.
Omitting the component of one $I_9$ fibre met by $Q$ and the component of the other $I_9$ fibre met by $\boxminus Q$,
we find two disjoint configurations of type $\te_8(-1)$
that are interchanged by $\jj$.
The lemma now follows from \cite[Thm.~5.7]{Mo}.
\end{proof}

The induced elliptic fibration on $X'$ has singular fibres $I_1, I_1, I_1, I_9, I_0^*, I_0^*$.
Since $\rho(X)=20$, both $X$ and $X'$ have $\MW$-rank two.
In particular, there are plenty of $\imath^*$-anti-invariant sections on $X$ (induced from $X'$).
As in \ref{ss:HS}, each such section gives an involution $\tau$.

\begin{Lemma}
There is a fixed-point free involution $\tau$ on $X$ as above.
\end{Lemma}

\begin{proof}
We verify the claim on $X'$ by assuming the contrary.
This means that for one of the $I_0^*$ fibres all non-identity components are avoided by $\MW(X')$.
As in \ref{ss:2}, this implies that $X'$ has discriminant four times an odd integer.
But we have seen that $X'$ has $T(X')=T(X)(2)$ with discriminant $-48$.
This gives a contradiction.
%
\end{proof}


\begin{Remark}
This example also shows that not every singular Enriques surface arises by the canonical Shioda-Inose structure from \ref{ss:SI2}.
This fact can also be seen in terms of Enriques surfaces with finite automorphism group.
Kond\=o classified these exceptional Enriques surfaces in \cite{K-E}.
Some are singular, but do not admit an elliptic fibration with a $II^*$ fibre.
\end{Remark}

\subsection{Brauer groups}
\label{ss:Brauer}

In \cite{HS}, we also answered a question by Beauville about Brauer groups.
Namely Beauville asked for explicit examples of complex Enriques surfaces $Y$ where the 
Brauer group $\Br(Y)\cong\Z/2\Z$ pulls back identically zero to the covering K3 surface $X$
via the universal cover $\pi: X \to Y$.
He also raised the question whether such an example exists over $\Q$.

In \cite[\S5]{HS}, we gave affirmative solutions for both questions.
Our basic objects were the singular K3 surfaces $X$ with
\begin{eqnarray}
\label{eq:M-N}
\NS(X) = U + 2E_8(-1) + \langle -4M\rangle + \langle -2N\rangle
\end{eqnarray}
where $M,N\in\N$ and $N>1$ is odd.
The above decomposition corresponds to an elliptic fibration \eqref{eq:Inose} on $X$ with $\MW$-rank two.
As in \ref{ss:HS}, the section $P$ of height $4M$ induces an Enriques involution $\tau$ on $X$.
Clearly the orthogonal section of height $2N$ gives an anti-invariant divisor for $\tau^*$.
By \cite{Beau}, this implies the vanishing of $\pi^*\Br(Y)$.

Previously we determined one surface (for $M=1, N=3$) with a model of \eqref{eq:Inose} and Enriques involution $\tau$ defined over $\Q$.
Here we want to point out that for any other surface $X$ as above,
this can be achieved over the class field $H(-8MN)$ by Theorem \ref{thm}.


\section{Classification problems}
\label{s:class}

We conclude this paper by formulating classification problems for singular Enriques surfaces.
In addition to fields of definition, we also consider Galois actions on divisors.
First we review the situation for singular K3 surfaces.

\subsection{Obstructions for singular K3 surfaces}
\label{ss:obs}

Although singular K3 surfaces can often be descended from the ring class field $H(d)$ to some smaller number field,
there are certain obstructions to this descent.
In this section we shall discuss two of them.
The first comes from the transcendental lattice.
Since the N\'eron-Severi lattice of a general K3 surface is determined by intersection numbers,
it is a geometric invariant,
i.e.~conjugate surfaces have the same $\NS$.
Since $T(X)$ and $\NS(X)$ are related as orthogonal complements in the K3 lattice $\Lambda$,
they share the same discriminant form up to sign by \cite[Prop.~1.6.1]{N}.
In particular, this fixes the genus of $T(X)$ (sometimes also called the isogeny class).

\begin{Theorem}[Shimada {\cite{Shimada-T}}, Sch\"utt {\cite{S-fields}}]
\label{Thm:genus}
Let $X$ be a singular K3 surface $X$ over some number field.
The transcendental lattices of $X$ and its Galois conjugates cover the full genus of $T(X)$.
\end{Theorem}

This result has an immediate consequence on the fields of definition:

\begin{Corollary}
Let $X$ be a singular K3 surface $X$ of discriminant $d$ over a number field $L$.
Let $K=\Q(\sqrt{-d})$ and $\bar L$ the Galois closure of $L$ over $K$.
Denote by $\mathcal G(X)$ the genus of $T(X)$.
Then 
\[
 \#\mathcal G(X) \mid \deg_K L.
\]
\end{Corollary}
In particular, one deduces that a singular K3 surface $X$ can only be  defined over $\Q$ 
if the genus of $T(X)$ consists of a single class.


The second obstruction stems from the Galois action on the divisors.
Namely, even if a singular K3 surface $X$ admits a model over a smaller field than $H(d)$,
the ring class field is preserved through the Galois action on $\NS(X)$:

\begin{Theorem}[Sch\"utt {\cite{S-NS}}]
\label{Thm:NS}
Let $X$ be a singular K3 surface of discriminant $d$ over some number field $L$. 
Assume that $\NS(X)$ is generated by divisors defined over $L$. 
Then the extension $L(\sqrt{d})$ contains the ring class field $H(d)$.
\end{Theorem}

In other words, Theorem \ref{Thm:NS-H(d)} is not far from being optimal:
at best, there is a model with $\NS(X)$ defined over a quadratic subfield of $H(d)$.


Theorem \ref{Thm:NS}  provides a direct proof of 
the following natural generalisation from CM elliptic curves (Shafarevich \cite{Shafa}):
Fixing $n\in\N$,
there are only finitely many singular K3 surfaces over all number fields of  degree bounded by $n$
(up to complex isomorphism).
The problem of explicit classifications, however, is still wide open.
Even in the simplest case, it is not clear yet how many singular K3 surfaces there are over $\Q$
-- only that there are many, cf.~\cite{ES}.
In contrast, the restrictive setting of Theorem \ref{Thm:NS} is much more accessible.
For instance there are exactly $13$ singular K3 surfaces up to $\bar\Q$-isomorphism with $\NS$ defined over $\Q$.
By \cite[Thm.~1]{S-NS}, they stand in bijective correspondence with the discriminants $d$ of class number one.

We shall now discuss how these obstructions turn out for singular Enriques surfaces.
Then we formulate analogous classification problems.

\subsection{Fields of definition of singular Enriques surfaces}
\label{ss:fields}

We start by pointing out that Theorem \ref{Thm:genus} carries over to singular Enriques surfaces directly.
This fact is due to the universal property that defines the covering K3 surface $X$ of an Enriques surface $Y$.
Explicitly, $X$ can be defined universally as
\[
X = \mbox{Spec}(\OO_Y \oplus \mathcal{K}_Y).
\]
As this construction respects the base field,
the obstructions from Theorem \ref{Thm:genus} on the field of definition of a singular K3 surface $X$
carry over to each singular Enriques surface that is covered by $X$.
Recall that a K3 surface may admit (arbitrarily) finitely many distinct Enriques quotients by \cite[Thm.~0.1]{Ohashi},
while the universal cover associates a unique K3 surface to a given Enriques surface.

\begin{Corollary}
Let $n\in\N$.
There are only finitely many singular Enriques surfaces 
over all number fields of degree at most $n$
up to complex isomorphism.
\end{Corollary}

\begin{Problem}
The following two questions concern singular Enriques surfaces up to $\bar\Q$-isomorphism:
\begin{enumerate}
 \item 
For $n\in\N$, find all singular Enriques surfaces over number fields $L$ of degree at most $n$ over $\Q$.
\item
Specifically classify all singular Enriques surfaces over $\Q$.
\end{enumerate}
\end{Problem}

\subsection{Galois action on divisors}
\label{ss:Gal}

Upon translating the obstructions for singular K3 surfaces from \ref{ss:obs} to singular Enriques surfaces,
we have seen in \ref{ss:fields} that Theorem \ref{Thm:genus} and its corollary carry over directly to the Enriques quotients.
In contrast, Theorem \ref{Thm:NS} has to be weakened on the Enriques side.
Generally speaking, this weakening is due to the fact that (part of) the Galois action can be accomodated by a sublattice of $\NS(X)$ that is killed by the Enriques involution.
In support of these ideas,
we shall review an example from \cite{HS} (that draws heavily from \cite{ES-19}).

Consider the following family $\mX$ of elliptic K3 surfaces
\begin{eqnarray}
\label{eq:mX}
\mathcal X:\;\; y^2 = x^3 + t^2 x^2 + t^3 (t-a)^2 x, \;\;\; a\neq 0.
\end{eqnarray}
This elliptic fibration has reducible singular fibres of type $III^*$ at $0$ and $\infty$ and $I_4$ at $t=a$.
The general member has Picard number $\rho(\mX)=19$ with 
\[
\MW(\mX)=\{O, (0,0)\}\cong\Z/2\Z.
\]
Note that $\mX$ is of base change type -- apply the base change $s=(t-a)^2/t$ to the rational elliptic surface $S$ with Weierstrass form
\[
 S:=\;\; y^2 = x^3 + x^2 + sx.
\]
As in \ref{ss:HS}, the two-torsion section induces an Enriques involution $\tau$ (unless the other singular fibres degenerate, i.e.~unless $a=-1/16$).
Denote the family of Enriques quotients by $\mY$.
We first study the Galois action on $\Num(\mY)$:

\begin{Lemma}
\label{Lem:Y_a}
Let $Y_a\in\mY \;(a\neq -1/16)$.
Then $\Num(Y_a)$ is defined over $\Q(a)$.
\end{Lemma}

\begin{proof}
Since $\Num(Y_a)$ is torsion-free,
the Galois action on $\Num(Y_a)$ coincides with that on the invariant part of $\NS(X_a)$.
In the present situation, 
the $I_4$ fibre of $\mX$ is split-multiplicative, i.e.~all fibre components are defined over $\Q(a)$.
The same holds trivially for the fibres of type $III^*$.
Together with the sections $O$ and $(0,0)$,
these rational curves generate $\NS(X_a)^{\tau^*}$ up to finite index.
As this holds regardsless of the Picard number of $X_a$ (being $19$ or $20$),
the lemma follows.
\end{proof}

\begin{Remark}
It is crucial that the lemma holds for \emph{all} members of the family $\mY$, i.e.~also the singular ones.
Compare the situation for singular K3 surfaces in the family $\mX$ where Theorem \ref{Thm:NS} will often enforce a Galois action on the additional generator of  $\NS$.
For the specialisations over $\Q$ with $\rho=20$, see \ref{ss:CM}.
\end{Remark}

\subsection{N\'eron-Severi group}
\label{ss:NS}

We point out that in this specific setting,
Lemma \ref{Lem:Y_a} gives a stronger statement than Corollary \ref{Cor:Num}.
The situation gets more complicated if we consider $\NS(\mY)$ with its two-torsion
because this can admit a quadratic Galois action.
In particular, we can only conjecture an analogue of Corollary \ref{Cor:Num} for $\NS(Y)$
that is more precise than saying that $\NS(Y)$ is defined over some quadratic extension of $H(d)$
(Conjecture \ref{conj}).

The main problem here lies in similar subtleties as encountered in the context of 
cohomologically and numerically trivial involutions (see \cite[\S4]{HS} and the references therein).
Namely, to decide about $\NS(Y)$ it is necessary to work out generators of the full group (see Remark \ref{Rem:NS}).
We work this out for the family $\mY$ in detail:

\begin{Proposition}
\label{Prop:NS}
Let $Y_a\in\mY \;(a\neq -1/16)$.
Then $\NS(Y_a)$ is defined over $\Q(a, \sqrt{-a})$.
\end{Proposition}

\begin{proof}
The next remark will indicate that it is not sufficient to argue with the elliptic fibration \eqref{eq:mX} on $\mX$.
Instead, we consider Inose's fibration \eqref{eq:Inose} for the given family.
The following Weierstrass form was derived in \cite[\S5.3]{HS}:
\[
\mathcal X:\;\; 
y'^2 = x'^3+(9a-1)x'/9+\left(27\left(u-\frac{a^3}u\right)+81a+2\right)/27.
\]
There is a section $P$ of height $4$ (thus disjoint from the zero section) with $x'$-coordinate
\[
P_{x'}(u) = (3  u^4+12  u^3 a+6  u^2 a^3+4  u^2 a^2-12  u a^4+3 a^6)/(12 a^2u^2).
\]
The section $P$ is anti-invariant for the base change involution $\imath$ of the Shioda-Inose structure on $\mX$:
\[
\imath: (x',y',u) \mapsto (x',y',-a^3/u).
\]
The base change involution composed with translation by $P$ defines an Enriques involution $\tau'$ on $\mX$ by \ref{ss:HS}.
Denote the family of Enriques quotients by $\mY'$.
By Kond\=o's classification in \cite{K-E},
$\mY'$ has finite automorphism group, and
in particular $\tau$ and $\tau'$ are conjugate in $\Aut(\mX)$ so that $\mY\cong \mY'$.

We continue by determining an explicit basis of $\NS(\mY')$.
The induced elliptic fibration on $\mY'$ has a singular fibre of type $II^*$,
a bisection $R$ (the push-down of $O$ and $P$) and two multiple smooth fibres $F_1=2G_1, F_2=2G_2$.
We claim that these twelve curves generate $\NS(\mY')$.
To see this, note that by construction $R$ meets the simple component of the $II^*$ fibre twice.
The remaining fibre components form the root lattice of type $E_8(-1)$.
Orthogonally in $\NS(\mY')$, we find $R, G_1, G_2$.
Since $R^2=-2, R\cdot G_i=1$, we know that $R, G_1$ generate the hyperbolic plane $U$.
Thus we have determined a unimodular lattice $L=U+E_8(-1)$ inside $\NS(\mY')$ -- necessarily of index two due to its rank being ten.
Since $G_2\not\in L$, it follows that $L$ and $G_2$ generate all of $\NS(\mY')$.

We now consider the Galois action on these generators of $\NS(Y_a')$ for some $Y_a'\in\mY'$.
Clearly the $II^*$ fibre and the bisection $R$ are defined over $\Q(a)$.
The multiple fibres sit at the ramification points of the base change on the base curve $\PP^1$, i.e.~at the roots of $u^2+a^3$.
Proposition \ref{Prop:NS} follows and cannot be improved since
the conjugation of $\Q(\sqrt{-a})/\Q(a)$ permutes the multiple fibres 
 if $\sqrt{-a}\not\in\Q(a)$,
and thus gives a non-trivial Galois action on $\NS(Y_a')$.
\end{proof}

\begin{Remark}
\label{Rem:NS}
Note that the above Galois action is not visible on the elliptic fibration \eqref{eq:mX} of $\mX$ yielding $\mY$.
The multiple fibres of the induced elliptic fibration on $\mY$ have different type $2I_0, 2I_2$.
Hence they cannot be interchanged by Galois.
Nonetheless there can be a nontrivial Galois action on $\NS(Y_a)$.
This goes undetected in the above model because the push-down of fibre components and torsion sections from $\mX$ to $\mY$ generate $\NS(\mY)$ only up to index two.
\end{Remark}

\subsection{CM-points} 
\label{ss:CM}

Concretely, the family $\mX$ is parametrised by the Fricke modular curve $X_0(2)^+$.
In \cite{ES-19}, we list all $\Q$-rational CM-points.
Two of them give singular K3 surfaces without Enriques involution (discriminant $-8$ at $a=-1/16$ and discriminant $-4$ at $a=0$ for a suitable alternative model of $\mX$).
The other 14 discriminants are:
\[
 -7, -12, -16, -20, -24, -28, -36, -40, -52, -72, -88, -100, -148, -232.
\]
For the discriminants of class number two,
the additional section can only be defined over a quadratic extension of $\Q$ by Theorem \ref{Thm:NS}.
So there are indeed singular Enriques surfaces with $\Num$ defined over $\Q$
where the same does not hold for the covering K3 surfaces.
A detailed example where this holds even for $\NS$ is provided by the surfaces at $a=-1/144$ 
which corresponds to the discriminant $-24$ (as mentioned in \ref{ss:Brauer}).
Details can be found in \cite[\S 5.3]{HS}.
We work out one example from the list 
where $\Num$ is defined over $\Q$, but $\NS$ is neither defined over $\Q$ nor over $H(d)$:

\begin{Example}
\label{Ex:Gal}
The specialisation $X$ with discriminant $d=-12$ sits at $a=1/9$.
In terms of the elliptic fibration \eqref{eq:mX}, there is a section of height $3$ over $H(d)=\Q(\sqrt{-3})$ with $x$-coordinate
$-12t^3/(9t-1)^2$.
One finds that $X$ has transcendental lattice two-divisible,
so $X$ is the Kummer surface of $E\times E$ for $E$ with j-invariant zero.
In particular $X$ is different from the singular K3  surface studied in Example \ref{Ex1} and \ref{ss:exc}.

The Enriques quotient $Y$ has multiple fibres at $\pm \sqrt{-1}/27$.
Compared with $\Num(Y)$ which is defined over $\Q$, complex conjugation acts on $\NS(Y)$ as non-trivial Galois action. 
Note that $H(d)(\sqrt{-1})=H(4d)$ in the present situation.
\end{Example}

\subsection{}

In the above example (and in fact for all specialisations over $\Q$ with $\rho=20$),
we have seen that $\NS(Y)$ is defined over the ring class field $H(4d)$.
We conjecture that this is always the case which would give an analogue of Corollary \ref{Cor:Num}:

\begin{Conjecture}
\label{conj}
Let $Y$ be an Enriques surface whose universal cover $X$ is a singular K3 surface.
Let $d<0$ denote the discriminant of $X$.
Then $Y$ admits a model over the ring class field $H(d)$ with $\NS(Y)$ defined over $H(4d)$.
\end{Conjecture}

The above one-dimensional family provides small evidence for this conjecture.
Our main motivation stems from the base change construction of Enriques involutions  in the framework of Shioda-Inose structures as investigated in Section \ref{s:Enr}.
By Proposition \ref{Prop:sing-Enr}, almost every possible singular K3 surface admits such an Enriques involution.
In terms of the model \eqref{eq:WF-X2},
the Enriques quotient $Y$ attains multiple fibres at the ramification points of the underlying base change, i.e.~at $\pm 2B$.
Recall from \eqref{eq:Inose} that $B^2=(1-j/12^3)(1-j'/12^3)$,
so there is a quadratic Galois action on $\NS(Y)$ unless $B\in H(d)$.
Note that $B$ can be interpreted in terms of  the Weber function $\sqrt{j-12^3}$ where $j$ now denotes the usual modular function.
The values of Weber functions at CM-point have been studied extensively starting from Weber.
In the present situation, Schertz proved that for singular $j$-values, $\sqrt{j-12^3}\in H(4d)$ \cite{Schertz}.
This implies:

\begin{Lemma}[Schertz]
\label{Lem:4d}
In the above setting, one has $B\in H(4d)$.
\end{Lemma}

We sketch an alternative proof of Lemma \ref{Lem:4d}.
It is based on a geometric approach that will also carry information about the Enriques surface $Y$ (and its elliptic fibration with fibre of type $II^*$).
Consider the Kummer surface $X'$ from the Shioda-Inose structure.
In general, it has fibres of type $I_0^*$ where $Y$ has the multiple fibres 
(if $E\cong E'$, there could be fibres of type $I_1^*$ or $IV^*$, cf.~\ref{ss:SI}).
By Corollary \ref{Cor:Km}, $X'$ has a model with $\NS(X')$ defined over $H(4d)$.
In particular, every elliptic fibration of $X'$ can be defined over $H(4d)$ with all of $\NS$ defined there as well.
We apply this argument to the elliptic fibration on $X'$  induced  from \eqref{eq:WF-X2}:
\begin{eqnarray}
\label{eq:X'}
\;\;\;\;\;\;\; X':\;\; y^2 = x^3 - 3c^2B^2A^3(t^2-4B^2)^2 x + c^3B^2A^3 (t-2B^2) (t^2-4B^2)^3.
\end{eqnarray}
Assume that $B\not\in H(d)$ and denote $L=H(d)(B)$.
By \cite{Sandwich}, the singular fibres of $X'$ predict  the Weierstrass form \eqref{eq:X'} (in case $AB\neq 0$) up to M\"obius transformation.
This property holds generally for constants $A,B,c$,
but in the present situation, $A$ and $B$ are related to the j-invariants of $E, E'$ by \eqref{eq:Inose}.
Upon applying M\"obius transformations, one can thus show that the above jacobian elliptic fibration does not admit a model over $H(d)$ \emph{without} $\Gal(L/H(d))$-action interchanging the $I_0^*$ fibres.
By Corollary \ref{Cor:Km}, one obtains that $B\in H(4d)$.
This proves Lemma \ref{Lem:4d}.

\begin{Corollary}
Conjecture \ref{conj} holds true for any singular Enriques surface arising from the Shioda-Inose structure as in Section \ref{s:Enr}.
\end{Corollary}

The geometric proof of Lemma \ref{Lem:4d}  is of particular interest to us, 
since the statement about the Galois action on the $I_0^*$ fibres of $X'$ 
carries over to the multiple fibres of the corresponding elliptic fibration of the Enriques surface $Y$
and vice versa.
Centrally, we use once again that a model of a K3 or Enriques surface with $\NS$ defined over a fixed field has \emph{all} elliptic fibrations (with or without section) defined over this field as well.
Hence we can move freely between models and elliptic fibrations.
Thus we obtain:

\begin{Corollary}
\label{Lem:B}
If $B\not\in H(d)$,
then any model over $H(d)$ of the Enriques surface $Y$ 
admits a non-trivial Galois action of $\Gal(H(4d)/H(d))$ on $\NS(Y)$.
\end{Corollary}

We have seen an instance of this phenomenon in Example \ref{Ex:Gal}.
The same reasoning implies a non-trivial action of $\Gal(\Q(a,\sqrt{-a})/\Q(a))$ on $\NS(Y_a)$ 
for all $\Q(a)$-models of members $Y_a$ of the family $\mY$.

The above results allow us to draw an analogy to the study of automorphisms of Enriques surfaces (cf.~\cite{BP}, \cite{MN}).
Namely we have exhibited two kind of singular Enriques surfaces over $H(d)$ -- one with cohomologically trivial Galois action and one with numerically, but not cohomologically trivial Galois action.




\subsection{}

We conclude this paper with the corresponding classification problem for singular Enriques surfaces.
Note that by the above reasoning, at least the second problem is more complicated than for K3 surfaces (as solved in \cite{S-NS}).

\begin{Problem}
The following two questions concern singular Enriques surfaces either up to $\bar\Q$- or up to $L$-isomorphism:
\begin{enumerate}
 \item 
For a given number field $L$ (or all number fields of bounded degree), classify all singular Enriques surfaces with $\Num$ or $\NS$ defined over $L$.
\item
Determine all singular Enriques surfaces over $L=\Q$ with trivial Galois action on $\Num$ or $\NS$.
\end{enumerate}
\end{Problem}


\begin{thebibliography}{99}


 \bibitem{BHPV}
 Barth, W., Hulek, K., Peters, C., van de Ven, A.:
 \emph{Compact complex surfaces}.
 Second edition,
 Erg.~der Math.~und ihrer Grenzgebiete,
 3.~Folge, Band {\bf 4}.~Springer (2004), Berlin.
 
 \bibitem{BP}
 Barth, W., Peters, C.:
 \emph{Automorphisms of Enriques surfaces},
 Invent.~Math.~{\bf 73} (1983), no.~3, 383--411.


\bibitem{Beau}
Beauville, A.:
\emph{On the Brauer group of Enriques surfaces},
 Math.~Res.~Lett.~{\bf 16}  (2009),  no.~6, 927--934.

\bibitem{BT}
Bogomolov, F., Tschinkel, Y.:
\emph{Density of rational points on Enriques surfaces},
Math.~Res.~Lett.~{\bf 5} (1998), no.~5, 623--628.

%
%

\bibitem{C} Cox, D.~A.: \emph{Primes of the Form $x^2+ny^2:$ Fermat, Class Field Theory, and  
Complex Multiplication}, Wiley Interscience
(1989).




\bibitem{ES} Elkies, N.~D., Sch\"utt, M.: \emph{Modular forms and K3 surfaces}, preprint (2008), arXiv: 0809.0830.


\bibitem{ES-19} Elkies, N.~D., Sch\"utt, M.:
\emph{K3 families of high Picard rank}, in preparation.



\bibitem{G} Gross, B.~H.: \emph{Arithmetic on elliptic curves with complex multiplication}. With an appendix by B.~Mazur. Lect.~Notes in Math.~{\bf 776}, Springer (1980).



\bibitem{HS}
Hulek, K., Sch\"utt, M.:
\emph{Enriques Surfaces and jacobian elliptic K3 surfaces},
to appear in Math.~Z.,
preprint (2009),
arXiv: 0912.0608.


\bibitem{Inose} Inose, H.:
 \emph{Defining equations of singular $K3$ surfaces and a notion of isogeny},
  in: \emph{Proceedings of the International Symposium on Algebraic Geometry}
  (Kyoto Univ., Kyoto, 1977),  Kinokuniya Book Store, Tokyo (1978), 495--502.

\bibitem{Keum}
Keum, J.~H.:
\emph{Every algebraic Kummer surface is the K3-cover of an Enriques surface},
Nagoya Math.~J.~{\bf 118}  (1990), 99--110.

%


\bibitem{K-E}
Kond\=o, S.: 
\emph{Enriques surfaces with finite automorphism groups},  
Japan.~J.~Math.~(N.S.)  {\bf 12 } (1986),  no.~2, 191--282.

%

%
%
%
%
%

\bibitem{Mo} Morrison, D.~R.: \emph{On K3 surfaces with large Picard number},  Invent.~Math.~{\bf 75}  (1984),  no.~1, 105--121.


%
 \bibitem{MN}
 Mukai, S., Namikawa, Y.:
 \emph{Automorphisms of Enriques surfaces which act trivially on the cohomology groups},
 Invent.~Math.~{\bf 77} (1984), no.~3, 383--397. 

\bibitem{N} Nikulin, V.~V.: \emph{Integral symmetric bilinear forms and some of their applications}, Math.~USSR Izv.~{\bf 14}, No.~1 (1980), 103-167.


\bibitem{Nishi} Nishiyama, K.-I.: \emph{The Jacobian fibrations on some K3 surfaces and their Mordell-Weil groups}, Japan.~J.~Math.~{\bf 22} (1996), 293--347.



 
 \bibitem{Ohashi}
 Ohashi, H.:
 \emph{On the number of Enriques quotients of a $K3$ surface},
 Publ.~RIMS  {\bf 43}  (2007),  no.~1, 181--200.


\bibitem{PSS} Piatetski-Shapiro , I.~I., Shafarevich, I.~R.: \emph{Torelli's theorem for algebraic surfaces of type ${\rm K}3$}, Izv.~Akad.~Nauk SSSR Ser.~Mat.~{\bf 35} (1971), 530--572.

\bibitem{Schertz}
Schertz, R.:
\emph{Die singul\"aren Werte der Weberschen Funktionen $\mathfrak f, \mathfrak f_{1},\mathfrak f_{2}, \gamma _{2}, \gamma _{3}$},
J.~Reine Angew.~Math.~{\bf 286/287}  (1976), 46--74.

\bibitem{S-fields} Sch\"utt, M.:
 \emph{Fields of definition of singular K3 surfaces},
 Communications in Number Theory and Physics {\bf 1}, 2 (2007), 307--321.

\bibitem{S-MMJ} Sch\"utt, M.:
 \emph{Arithmetic of a singular K3 surface},
Michigan Math.~J.~{\bf 56} (2008), 513--527.

%

\bibitem{S-NS} 
Sch\"utt, M.: \emph{K3 surfaces of Picard rank 20 over $\Q$}, Algebra \& Number Theory {\bf 4} (2010), no.~3, 335--356.
 
 \bibitem{SSh}
 Sch\"utt, M., Shioda, T.:
 \emph{Elliptic surfaces},
Algebraic geometry in East Asia - Seoul 2008, 
Advanced Studies in Pure Math.~{\bf 60} (2010), 51-160.  
%

\bibitem{Sert}
Sert\"oz, A.~S.:
\emph{Which singular K3 surfaces cover an Enriques surface},
Proc.~AMS {\bf 133} (2005), no.~1, 43--50

\bibitem{Shafa} Shafarevich, I. R: \emph{On the arithmetic of singular $K3$-surfaces}, in: \emph{Algebra and analysis} (Kazan 1994), de Gruyter (1996), 103--108.


\bibitem{Shimada-T} Shimada, I.:
 \emph{Transcendental lattices and supersingular reduction lattices
 of a singular $K3$ surface}, 
Trans.~Amer.~Math.~Soc.~{\bf 361}  (2009), 909--949. 


\bibitem{SZ} Shimada, I., Zhang, D.~Q.:
\emph{Classification of extremal elliptic $K3$ surfaces
and fundamental groups of open $K3$ surfaces},
Nagoya Math.~J.~{\bf 161} (2001), 23--54.


\bibitem{Shimura} Shimura, G.: \emph{Introduction to the arithmetic theory of automorphic functions}, Iwanami Shoten \& Princeton University Press (1971).



\bibitem{ShMW} Shioda, T.: \emph{On the Mordell-Weil lattices}, Comm.~Math.~Univ.~St.~Pauli {\bf 39} (1990), 211--240.


\bibitem{Sandwich} Shioda, T.:
\emph{Kummer sandwich theorem of certain elliptic K3 surfaces},
Proc.~Japan~Acad.~{\bf 82}, Ser.~A (2006), 137--140.


\bibitem{Sh-Murre}
Shioda, T.:
\emph{Correspondence of elliptic curves and Mordell-Weil lattices of certain elliptic K3 surfaces},
in \emph{Algebraic Cycles and Motives}, Vol.~2, 
LMS Lect.~Note Ser.~{\bf 344}, Cambridge Univ. Press (2007),  319--339.



\bibitem{SI} Shioda, T., Inose, H.:
 \emph{On Singular $K3$ Surfaces},
 in: Baily, W.~L.~Jr., Shioda, T.~(eds.),
 \emph{Complex analysis and algebraic geometry},
 Iwanami Shoten, Tokyo (1977), 119--136.


\bibitem{SM} Shioda, T., Mitani, N.:
 \emph{Singular abelian surfaces and binary quadratic forms},
 in: \emph{Classification of algebraic varieties
 and compact complex manifolds},
 Lect.~Notes in Math.~{\bf 412} (1974), 259--287.

\bibitem{Si} Silverman, J.~H.:
 \emph{Advanced Topics in the Arithmetic of Elliptic Curves.}
 Graduate Texts in Math.~{\bf 151}, Springer (1994).



\bibitem{Sterk}
Sterk, H.:
\emph{Finiteness results for algebraic $K3$ surfaces},
Math.~Z.~{\bf  189}  (1985),  no.~4, 507--513.

 
 
 \bibitem{Tate} Tate, J.: {\it Algorithm for determining the type
 of a singular fibre in an elliptic pencil}, in: {\it Modular
 functions of one variable IV} (Antwerpen 1972), SLN {\bf 476}
 (1975), 33--52.


\bibitem{Wb} Weinberger, P.~J.:
 \emph{Exponents of the class groups of complex quadratic fields},
 Acta Arith.~{\bf 22} (1973), 117--124.



\end{thebibliography}
\end{document}